\documentclass[11pt]{amsart}

\usepackage[margin=1in]{geometry}
\usepackage{amsmath,amssymb,amsthm,mathtools,mathrsfs,bm}
\usepackage{enumitem}

\usepackage[colorlinks=true, citecolor=blue, linkcolor=red, urlcolor=magenta]{hyperref}

\numberwithin{equation}{section}

\newtheorem{theorem}{Theorem}[section]
\newtheorem{lemma}[theorem]{Lemma}
\newtheorem{proposition}[theorem]{Proposition}
\newtheorem{corollary}[theorem]{Corollary}
\theoremstyle{definition}
\newtheorem{definition}[theorem]{Definition}
\newtheorem{remark}[theorem]{Remark}

\newcommand{\R}{\mathbb R}
\newcommand{\N}{\mathbb N}
\newcommand{\E}{\mathbb E}
\newcommand{\Var}{\operatorname{Var}}
\newcommand{\OT}{\operatorname{OT}}
\newcommand{\Lip}{\operatorname{Lip}}
\newcommand{\diam}{\operatorname{diam}}
\newcommand{\supp}{\operatorname{supp}}
\newcommand{\cP}{\mathcal P}
\newcommand{\cF}{\mathcal F}
\newcommand{\cG}{\mathcal G}
\newcommand{\1}{\mathbf 1}
\newcommand{\DGW}{\mathcal D}
\newcommand{\MGW}{\mathcal M}
\newcommand{\CGW}{\mathcal C}
\newcommand{\QGW}{\mathcal Q}
\newcommand{\Cov}{\mathcal N}
\newcommand{\Pair}[2]{\left\langle #1,#2\right\rangle}

\title{Empirical Convergence of Even-Order Gromov--Wasserstein Functionals}

\author{Vasyl Paliy}
\address{Department of Mathematics, Brown University, Providence, RI 02912, USA}
\email{vasyl\_paliy@brown.edu}

\date{\today}

\begin{document}

\setlength{\parindent}{0pt}

\begin{abstract}
We study the sample complexity of empirical plug-in estimation for the powered even-order
Gromov-Wasserstein functional between compactly supported probability measures on
\(\mathbb R^{d_x}\) and \(\mathbb R^{d_y}\). For every fixed pair of integers \(r,k\ge 1\), we prove
that the two-sample empirical error is bounded at the rate
\[
    n^{-2/\max\{\min\{d_x,d_y\},4\}},
\]
up to a logarithmic factor in the critical case \(\min\{d_x,d_y\}=4\). This extends the known quadratic Euclidean upper rate to the full powered even-order family.  The proof uses a polynomial decomposition of the
even-order GW functional, a generalized duality formula reducing the coupling-dependent
term to a compact family of ordinary optimal transport problems, and entropy estimates for
semiconcave dual potentials. 
\end{abstract}

\maketitle
\vspace{-1.5em}

\section{Introduction}

\subsection{Gromov--Wasserstein distance}

The Gromov--Wasserstein (GW) distance, introduced by M\'emoli \cite{Memoli}, extends optimal transport to the comparison of probability measures supported on different metric spaces. In contrast with classical optimal transport, which requires a cost function defined directly on the product of the two spaces, the GW framework compares the internal pairwise distance structures of the underlying spaces. Given metric measure spaces $(X,d_X,\mu)$ and $(Y,d_Y,\nu)$, the $(p,q)$-GW distance is defined by
\[
D_{p,q}(\mu,\nu)
:=
\inf_{\pi\in\Pi(\mu,\nu)}
\left(
\iint
\bigl|
d_X(x,x')^q-d_Y(y,y')^q
\bigr|^p
\,d\pi(x,y)\,d\pi(x',y')
\right)^{1/p}.
\]
This may be viewed as an $L^p$ relaxation of the Gromov--Hausdorff distance, and it defines a metric on isomorphism classes of metric measure spaces; see \cite{Memoli,Sturm}.

\medskip

\noindent In this paper, we work in the Euclidean setting, where
\[
\mathcal X \subset \mathbb{R}^{d_x},
\qquad
\mathcal Y \subset \mathbb{R}^{d_y}
\]
are compact, and we study the even-order family
\[
\DGW_{r,k}(\mu,\nu)
:=
\inf_{\pi\in\Pi(\mu,\nu)}
\iint
\bigl(
\|x-x'\|^{2k}-\|y-y'\|^{2k}
\bigr)^{2r}
\,d\pi(x,y)\,d\pi(x',y'),
\]
where $r,k\geq 1$ are integers. The even exponents guarantee non-negativity of the integrand and allow a polynomial expansion of the kernel.

\medskip

A basic statistical question is how well $D_{p,q}(\mu,\nu)$ is approximated by its empirical plug-in version. Unlike classical optimal transport, however, the GW functional is quadratic in the coupling, since the same coupling appears in both integration variables. As a result, standard Kantorovich duality does not apply directly; compare \cite{Villani}. This makes quantitative analysis of empirical GW distances substantially more delicate than in the classical setting. 

\smallskip

A major recent advance was obtained by Zhang, Goldfeld, Mroueh, and Sriperumbudur \cite{ZhangEtAl}. In the quadratic Euclidean case $(2,2)$, for compactly supported measures on $\mathbb{R}^{d_x}$ and $\mathbb{R}^{d_y}$, they established 
\[
\mathbb{E}\Bigl|
D_{2,2}(\mu,\nu)^2-D_{2,2}(\widehat{\mu}_n,\widehat{\nu}_n)^2
\Bigr|
=
\begin{cases}
O(n^{-1/2}), & d_x\wedge d_y<4,\\
O(n^{-1/2}\log n), & d_x\wedge d_y=4,\\
O(n^{-2/(d_x\wedge d_y)}), & d_x\wedge d_y>4,
\end{cases}
\]
and showed that this rate is sharp up to the logarithmic factor at the critical dimension. They also considered higher powers and observed that, for the standard $(2,2k)$-GW distance, their direct argument gives the slower upper bound
\[
O\!\left(n^{-1/(d_x\wedge d_y)}\right).
\]
In the present paper, we show that for every fixed pair of integers $r,k\ge1$, the empirical plug-in estimator for the powered functional $D_{2r,2k}^{2r}$ satisfies the same dimension-dependent upper bound as in the quadratic case. The main point is that, after the polynomial dual reduction, the relevant optimal potentials still belong to uniformly semiconcave classes on bounded supports, which gives the sharper Bronshtein-type entropy bound.

\subsection{Notation}
 
We write $\|\cdot\|$ for the Euclidean norm on $\mathbb{R}^d$ and
$\langle \cdot, \cdot \rangle$ for the standard inner product.
For $x \in \mathbb{R}^d$ and $r > 0$, let $B_d(x,r) := \{y \in \mathbb{R}^d : \|y - x\| \leq r\}$
denote the closed Euclidean ball of radius $r$ centered at $x$.
Given a Borel set $E \subseteq \mathbb{R}^d$, we write $\mathcal{P}(E)$ for the collection of
Borel probability measures supported on $E$,
and for $m \geq 1$ we let $\mathcal{P}_m(\mathbb{R}^d)$ denote the subset of Borel probability
measures on $\mathbb{R}^d$ with finite $m$-th moment, i.e.,
$\int_{\mathbb{R}^d} \|x\|^m \, d\mu(x) < \infty$.
For $\mu \in \mathcal{P}(\mathbb{R}^{d_x})$ and $\nu \in \mathcal{P}(\mathbb{R}^{d_y})$,
we denote by $\Pi(\mu,\nu)$ the set of all couplings of $\mu$ and $\nu$,
that is, all Borel probability measures $\pi$ on $\mathbb{R}^{d_x} \times \mathbb{R}^{d_y}$
whose first and second marginals are $\mu$ and $\nu$, respectively.
If $T : \mathbb{R}^d \to \mathbb{R}^{d'}$ is a measurable map and $\mu \in \mathcal{P}(\mathbb{R}^d)$,
we write $T_{\#}\mu$ for the pushforward measure defined by $T_{\#}\mu(A) = \mu(T^{-1}(A))$.

\medskip

For a signed Borel measure $\sigma$ and a bounded measurable function $h$, write
\[
\Pair{\sigma}{h}:=\int h\,d\sigma
\]
whenever the integral is well defined.
The support of a measure $\mu \in \mathcal{P}(\mathbb{R}^d)$ is denoted $\mathrm{supp}(\mu)$,
and its diameter is $\mathrm{diam}(S) := \sup_{x,y \in S} \|x - y\|$.
Given samples $X_1, \dots, X_n$ drawn independently from $\mu$, the empirical measure is
$\widehat{\mu}_n := \frac{1}{n} \sum_{i=1}^n \delta_{X_i}$,
where $\delta_x$ is the Dirac mass at $x$.

\medskip

For a function $f : D \to \mathbb{R}$ on a domain $D \subseteq \mathbb{R}^d$, we write
$\|f\|_{L^\infty(D)} := \sup_{x \in D} |f(x)|$ for the supremum norm and
$\|f\|_{\mathrm{Lip}(D)} := \sup_{x \neq y} \frac{|f(x) - f(y)|}{\|x - y\|}$
for the Lipschitz seminorm.
For a function class $\mathcal{F}$ and a metric $d$, we denote by
$\Cov(\varepsilon, \mathcal{F}, d)$ the $\varepsilon$-covering number of $\mathcal{F}$
with respect to $d$, i.e., the smallest number of $d$-balls of radius $\varepsilon$
needed to cover $\mathcal{F}$.

\medskip
 
For a multi-index $\alpha = (\alpha_1, \dots, \alpha_d) \in \mathbb{N}_0^d$,
we write $|\alpha| = \alpha_1 + \cdots + \alpha_d$ and
$x^\alpha = x_1^{\alpha_1} \cdots x_d^{\alpha_d}$.
For a coupling $\pi\in \Pi(\mu,\nu)$ and multi-indices $\alpha,\gamma$, write
\[
M_{\alpha,\gamma}(\pi):=\int x^\alpha y^\gamma\,d\pi(x,y),
\]
and for marginal moments write
\[
M^X_\alpha(\mu):=\int x^\alpha\,d\mu(x),
\qquad
M^Y_\gamma(\nu):=\int y^\gamma\,d\nu(y).
\]
For the even-order GW functional, set
\[
\DGW_{r,k}(\mu,\nu):=D_{2r,2k}(\mu,\nu)^{2r}.
\]
We also use
\[
d_\ast:=d_x\wedge d_y,
\qquad
\rho_n(d):=n^{-2/(d\vee 4)}\bigl(\log(en)\bigr)^{\1_{\{d=4\}}}.
\]
We use $a \vee b := \max\{a,b\}$ and $a \wedge b := \min\{a,b\}$ for the maximum and minimum
of two real numbers. Throughout, we use $C$ and $C_{r,k,d_x,d_y}$ (and similar subscripted variants) to denote
positive constants whose values may change from line to line but depend only on the indicated parameters.

\subsection{Main Result}

Our main theorem shows that the empirical rate is governed by the smaller of the two ambient
dimensions. This is in the same spirit as the lower-complexity phenomenon for classical optimal
transport between measures on different spaces proved by Hundrieser, Staudt, and Munk
\cite{HundrieserStaudtMunk}.

\begin{theorem}\label{thm:main}
Let $\mu\in \cP(\mathcal X)$ and $\nu\in \cP(\mathcal Y)$, where
\[
    \mathcal X\subset \R^{d_x},
    \qquad
    \mathcal Y\subset \R^{d_y}
\]
are compact. Set
\[
    R:=\diam(\mathcal X)\vee \diam(\mathcal Y),
    \qquad
    d_\ast:=d_x\wedge d_y.
\]
By translation invariance, we may assume
\[
    \mathcal X\subset B_{d_x}(0,R),
    \qquad
    \mathcal Y\subset B_{d_y}(0,R).
\]
Let
\[
    \widehat\mu_n=\frac1n\sum_{i=1}^n\delta_{X_i},
    \qquad
    \widehat\nu_n=\frac1n\sum_{j=1}^n\delta_{Y_j},
\]
where $(X_i)_{i\ge1}$ are i.i.d.\ with law $\mu$, $(Y_j)_{j\ge1}$ are i.i.d.\ with law $\nu$, and the two
samples are independent. Then, for every pair of integers $r,k\ge1$,
\[
\E\Bigl|
\DGW_{r,k}(\mu,\nu)
-
\DGW_{r,k}(\widehat\mu_n,\widehat\nu_n)
\Bigr|
\lesssim_{r,k,d_x,d_y}
R^{4kr}
\begin{cases}
n^{-1/2}, & d_\ast<4,\\[2mm]
n^{-1/2}\log(en), & d_\ast=4,\\[2mm]
n^{-2/d_\ast}, & d_\ast>4.
\end{cases}
\]
Equivalently, if
\[
    \rho_n(d):=
    n^{-2/(d\vee 4)}\log(en)^{\mathbf 1_{\{d=4\}}},
\]
then
\[
\E\Bigl|
\DGW_{r,k}(\mu,\nu)
-
\DGW_{r,k}(\widehat\mu_n,\widehat\nu_n)
\Bigr|
\lesssim_{r,k,d_x,d_y}
R^{4kr}\rho_n(d_\ast).
\]
\end{theorem}

\bigskip

\begin{remark}
For $(r,k)=(1,1)$, Theorem~\ref{thm:main} matches the known sharp quadratic GW rate established in \cite{ZhangEtAl}. More generally, it shows that the same upper-bound scale driven by the smaller ambient dimension persists throughout the even-order family $(2r,2k)$. For $k\ge2$, this improves the rate obtained by the direct Lipschitz-potential argument in \cite{ZhangEtAl}, which gives $O(n^{-1/(d_x\wedge d_y)})$ for the standard $(2,2k)$ case.  However, unlike the quadratic case treated in \cite{ZhangEtAl}, 
we only prove a matching parametric lower bound in the
case \(d_x\wedge d_y<4\). Matching lower bounds in the critical and high-dimensional regimes \(d_x\wedge d_y\ge 4\) remain open for the general even-order family considered here.
\end{remark}

\subsection{Proof Strategy}

We outline the proof of Theorem~\ref{thm:main}. The proof proceeds as follows.
\begin{enumerate}[label=(\roman*), leftmargin=2.5em]
\item \emph{Normalize the geometry.}
Using translation invariance and homogeneity of the GW functional, we reduce to the case where
both measures are centered and supported in fixed unit balls. The original scale is recovered at
the end, which is how we get the factor $R^{4kr}$.

\item \emph{Center the empirical measures.}
The empirical measures are replaced by their centered versions. Since GW is invariant under
independent translations of the two spaces, this does not change the GW value itself. The only
additional bookkeeping comes from the polynomial marginal terms, and these errors are controlled
at the parametric rate $n^{-1/2}$.

\item \emph{Separate marginal and coupling dependence.}
Expanding the even polynomial kernel gives a decomposition
\[
\mathcal \DGW_{r,k}(\mu,\nu)
=
\mathcal M_{r,k}(\mu,\nu)
+
\mathcal C_{r,k}(\mu,\nu).
\]
Here $\mathcal M_{r,k}$ depends only on finitely many marginal moments of $\mu$ and $\nu$,
and $\mathcal C_{r,k}$ contains all terms that depend on the coupling.

\item \emph{Control the marginal part.}
Because $\mathcal M_{r,k}$ is a fixed polynomial in finitely many moments, its empirical error is
bounded by standard moment estimates. This gives an $O(n^{-1/2})$ contribution, which is
absorbed by $\rho_n(d_\ast)$.

\item \emph{Convert the coupling part to ordinary OT.}
The duality argument rewrites $\mathcal C_{r,k}$ as a minimax expression over a compact
parameter set:
\[
\mathcal C_{r,k}(\mu,\nu)
=
4\inf_{v\in\Theta_-}\sup_{u\in\Theta_+}
\left\{
-\|u\|^2+\|v\|^2+\operatorname{OT}_{c_{u,v}}(\mu,\nu)
\right\},
\]
where $\{c_{u,v}\}$ is a finite-dimensional family of polynomial costs. The compactness of
$\Theta_+\times\Theta_-$ is uniform over the normalized measures, so the population and empirical
problems are compared using the same cost family.

\item \emph{Reduce to empirical processes over dual potentials.}
For each fixed polynomial cost $c_\theta$, Kantorovich duality bounds
\[
\left|
\operatorname{OT}_{c_\theta}(\mu,\nu)
-
\operatorname{OT}_{c_\theta}(\widetilde\mu_n,\widetilde\nu_n)
\right|
\]
by empirical-process suprema over first and second dual potentials. The first potentials lie in a
uniformly bounded, uniformly Lipschitz, uniformly semiconcave class $\mathcal F_X$. The second
potentials are obtained by $c_\theta$-transforms and therefore lie in
\[
\mathcal G_Y
:=
\bigcup_{\theta\in\Theta}
\mathsf C_\theta(\mathcal F_X).
\]

\item \emph{Apply entropy bounds and use the lower dimension.}
Without loss of generality, assume $d_x=d_\ast\le d_y$. The semiconcavity of
$\mathcal F_X$ gives the entropy estimate
\[
\log \mathcal N(\varepsilon,\mathcal F_X,\|\cdot\|_\infty)
\lesssim
\varepsilon^{-d_\ast/2},
\]
via Bronshtein-type covering bounds for convex functions. The class $\mathcal G_Y$ is controlled
by combining this entropy estimate with a finite-dimensional covering of the compact parameter
set $\Theta$. Dudley's entropy integral then gives
\[
\E \sup_{f\in\mathcal F_X}
\left|\langle \mu-\widetilde\mu_n,f\rangle\right|
+
\E \sup_{g\in\mathcal G_Y}
\left|\langle \nu-\widetilde\nu_n,g\rangle\right|
\lesssim
\rho_n(d_\ast).
\]
which proves the theorem. 
\end{enumerate}

The rest of the paper implements these steps. Section~\ref{sec:setup} collects the required
preliminaries: semiconcavity entropy bounds, the polynomial decomposition of the GW kernel, and
the generalized dual representation of the coupling-dependent term. Section~\ref{sec:proof-main}
proves Theorem~\ref{thm:main}. Appendix~\ref{app:duality} proves the generalized duality formula,
Appendix~\ref{app:marginal} contains the polynomial bookkeeping for the marginal term, and
Appendix~\ref{app:covering} proves the auxiliary covering estimates for the second-potential class.

\medskip

\section{Background and preliminaries}\label{sec:setup}

\subsection{Classical optimal transport}\label{sec:classical-ot}

We briefly review the classical optimal transport problem;
for a thorough treatment see~\cite{Villani}.
Let $\mathcal{X}$, $\mathcal{Y}$ be Polish spaces and let
$c : \mathcal{X} \times \mathcal{Y} \to \mathbb{R}$ be a lower semicontinuous
cost function.
The optimal transport cost between $\mu \in \mathcal{P}(\mathcal{X})$ and
$\nu \in \mathcal{P}(\mathcal{Y})$ with respect to $c$ is
\begin{equation}\label{eq:ot-primal}
  \OT_c(\mu, \nu)
  := \inf_{\pi \in \Pi(\mu,\nu)} \int_{\mathcal{X} \times \mathcal{Y}}
  c(x,y) \, d\pi(x,y).
\end{equation}
The special case $\mathcal{X} = \mathcal{Y} = \mathbb{R}^d$ with
$c(x,y) = \|x - y\|^p$, $p \geq 1$, gives rise to the $p$-Wasserstein distance
$W_p(\mu,\nu) := \OT_c(\mu,\nu)^{1/p}$, which metrizes weak convergence
together with convergence of $p$-th moments on $\mathcal{P}_p(\mathbb{R}^d)$.

\medskip

Suppose that $c(x,y) \geq a(x) + b(y)$ for some upper semicontinuous functions
$(a,b) \in L^1(\mu) \times L^1(\nu)$.
Then OT admits the Kantorovich dual representation~\cite{Villani}:
\begin{equation}\label{eq:ot-dual}
  \OT_c(\mu, \nu)
  = \sup_{\substack{(\varphi, \psi) \in C_b(\mathcal{X}) \times C_b(\mathcal{Y}) \\
  \varphi(x) + \psi(y) \leq c(x,y)}}
  \left\{ \int \varphi \, d\mu + \int \psi \, d\nu \right\}.
\end{equation}
Given $\varphi \in C_b(\mathcal{X})$, the $c$-transform
$\varphi^c(y) := \inf_{x \in \mathcal{X}} \{ c(x,y) - \varphi(x) \}$
and, symmetrically, the $\bar{c}$-transform
$\psi^{\bar{c}}(x) := \inf_{y \in \mathcal{Y}} \{ c(x,y) - \psi(y) \}$
provide dual feasible pairs, and the supremum in~\eqref{eq:ot-dual} can be restricted
to pairs satisfying $\psi = \varphi^c$ and $\varphi = \psi^{\bar{c}}$
without changing the value.

\medskip

The dual form is the key link to empirical process theory.
If $\mathcal{F}$ and $\mathcal{G}$ are function classes containing all optimal
first and second potentials, then
\begin{equation}\label{eq:ot-empirical-reduction}
  \bigl| \OT_c(\mu, \nu) - \OT_c(\mu', \nu') \bigr|
  \leq \sup_{\varphi \in \mathcal{F}} \bigl|\Pair{\mu - \mu'}{\varphi}\bigr|
  + \sup_{\psi \in \mathcal{G}} \bigl|\Pair{\nu - \nu'}{\psi}\bigr|.
\end{equation}
When $(\mu', \nu') = (\widehat{\mu}_n, \widehat{\nu}_n)$, the right-hand side
consists of suprema of empirical processes whose convergence rate
is governed by the metric entropy of $\mathcal{F}$ and $\mathcal{G}$
via Dudley's entropy integral inequality~\cite{Dudley}.
Regularity of optimal potentials, i.e. Lipschitz continuity
and semiconcavity inherited from the cost, gives covering number
bounds that determine the rate;
see~\cite{Dudley, ManoleNilesWeed, HundrieserStaudtMunk} for the
development of this program in the classical OT setting.
We use similar strategies to prove the main result.

\subsection{Semiconcavity and covering numbers}

\begin{definition}\label{def:semicolon}
Let $D\subset \R^d$ be convex. A function $f:D\to \R$ is \emph{$\Lambda$-semiconcave} if
\[
x\longmapsto f(x)-\frac{\Lambda}{2}\|x\|^2
\]
is concave on $D$. If $S\subset D$ is arbitrary, we say that $f:S\to \R$ is \emph{$\Lambda$-semiconcave on
$S$ relative to $D$} if it admits an extension to $D$ that is $\Lambda$-semiconcave there.
\end{definition}

\begin{lemma}\label{lem:semicolon-entropy}
Let $D=B_d(0,R_0)$. Fix $\Lambda,L_{\mathrm{sc}},C_{\mathrm{sup}}>0$, and define
\[
\cF_{\mathrm{sc}}(D;\Lambda,L_{\mathrm{sc}},C_{\mathrm{sup}})
:=
\Bigl\{
f:D\to \R:
f \text{ is }\Lambda\text{-semiconcave},
\ \|f\|_{\Lip(D)}\le L_{\mathrm{sc}},
\ \|f\|_{L^\infty(D)}\le C_{\mathrm{sup}}
\Bigr\}.
\]
Then
\[
\log \Cov\bigl(\varepsilon,\cF_{\mathrm{sc}}(D;\Lambda,L_{\mathrm{sc}},C_{\mathrm{sup}}),\|\cdot\|_\infty\bigr)
\le
C_d\,\varepsilon^{-d/2},
\]
where $C_d$ depends only on $d,\Lambda,L_{\mathrm{sc}},C_{\mathrm{sup}},R_0$.
\end{lemma}

\begin{proof}
Define $T:f\mapsto g$ by
\[
g(x):=\frac{\Lambda}{2}\|x\|^2-f(x).
\]
If $f$ is $\Lambda$-semiconcave, then $g$ is convex. Moreover,
\[
\|g\|_{L^\infty(D)}
\le
C_{\mathrm{sup}}+\frac{\Lambda R_0^2}{2}
=:C_{\mathrm{conv}},
\qquad
\|g\|_{\Lip(D)}
\le
L_{\mathrm{sc}}+\Lambda R_0
=:L_{\mathrm{conv}}.
\]
Since $T$ is an isometry for $\|\cdot\|_\infty$,
\[
\Cov\bigl(\varepsilon,\cF_{\mathrm{sc}}(D;\Lambda,L_{\mathrm{sc}},C_{\mathrm{sup}}),\|\cdot\|_\infty\bigr)
=
\Cov\bigl(\varepsilon,T(\cF_{\mathrm{sc}}(D;\Lambda,L_{\mathrm{sc}},C_{\mathrm{sup}})),\|\cdot\|_\infty\bigr).
\]
Also,
\[
T(\cF_{\mathrm{sc}}(D;\Lambda,L_{\mathrm{sc}},C_{\mathrm{sup}}))
\subset
\cF_{conv}(D;C_{\mathrm{conv}},L_{\mathrm{conv}}),
\]
where
\[
\cF_{conv}(D;C_{\mathrm{conv}},L_{\mathrm{conv}})
:=
\Bigl\{
g:D\to \R:
g \text{ is convex},
\ \|g\|_{L^\infty(D)}\le C_{\mathrm{conv}},
\ \|g\|_{\Lip(D)}\le L_{\mathrm{conv}}
\Bigr\}.
\]
Bronshtein's entropy bound for convex functions \cite{Bronshtein}; see also
\cite{GaoWellner}; therefore gives
\[
\log \Cov\bigl(\varepsilon,\cF_{\mathrm{sc}}(D;\Lambda,L_{\mathrm{sc}},C_{\mathrm{sup}}),\|\cdot\|_\infty\bigr)
\le
\log \Cov\bigl(\varepsilon,\cF_{conv}(D;C_{\mathrm{conv}},L_{\mathrm{conv}}),\|\cdot\|_\infty\bigr)
\le
C_d\,\varepsilon^{-d/2}.
\]
\end{proof}

\subsection{Decomposition into a marginal term and a transport term}

For later use we expand the GW kernel and postpone the detailed bookkeeping to
Appendix~\ref{app:marginal}. Set
\[
K_X(x,x'):=\|x-x'\|^{2k},
\qquad
K_Y(y,y'):=\|y-y'\|^{2k}.
\]
Then
\[
\bigl(K_X(x,x')-K_Y(y,y')\bigr)^{2r}
=
\sum_{s=0}^{2r}(-1)^s\binom{2r}{s}\bigl(K_X(x,x')\bigr)^s\bigl(K_Y(y,y')\bigr)^{2r-s}.
\]
After expanding $\bigl(K_X^s(x,x')\bigr)^s$ and $\bigl(K_Y(y,y')\bigr)^{2r-s}$ by the multinomial theorem, every term in the kernel becomes a product
of two mixed moments of the form
\[
M_{\alpha,\gamma}(\pi)\,M_{\beta,\delta}(\pi).
\]

\begin{definition}\label{def:marginal-only}
A term in this expansion is called \emph{marginal-only} if each of its two factors depends only on the
marginals of $\pi$; equivalently,
\[
(\alpha=0 \text{ or } \gamma=0)
\quad\text{and}\quad
(\beta=0 \text{ or } \delta=0).
\]
Let $\MGW_{r,k}(\mu,\nu)$ be the sum of all marginal-only terms. Let $\QGW_{r,k}(\pi)$ denote the sum
of the remaining terms, and define
\[
\CGW_{r,k}(\mu,\nu)
:=
\inf_{\pi\in \Pi(\mu,\nu)}\QGW_{r,k}(\pi).
\]
\end{definition}

By construction,
\begin{equation}\label{eq:marginal-coupling-decomp}
\DGW_{r,k}(\mu,\nu)
=
\MGW_{r,k}(\mu,\nu)+\CGW_{r,k}(\mu,\nu).
\end{equation}
The detailed polynomial description of $\MGW_{r,k}$ appears in Appendix~\ref{app:marginal}. The main point
for the body of the paper is that $\MGW_{r,k}$ depends only on finitely many marginal moments, while
$\CGW_{r,k}$ carries all of the coupling dependence.

\subsection{Generalized duality and the bounded-support reduction}

The following dual representation is the even-order analogue of the polynomial duality mechanism in \cite{ZhangEtAl}. We include the proof in Appendix~\ref{app:duality} to make the dependence on $(r,k)$ and the bounded-support parameter sets explicit.

\medskip

For any measurable cost $c:\R^{d_x}\times \R^{d_y}\to \R$ that is integrable against every
$\pi\in \Pi(\mu,\nu)$, set
\[
\OT_c(\mu,\nu)
:=
\inf_{\pi\in \Pi(\mu,\nu)}\int c\,d\pi.
\]

\begin{theorem}[Generalized duality for $(2r,2k)$]\label{thm:duality-main}
Let $r,k\in \N$, and let
\[
\mu\in \cP_{4kr}(\R^{d_x}),
\qquad
\nu\in \cP_{4kr}(\R^{d_y}).
\]
Then there exist
\begin{itemize}[leftmargin=2em]
\item[a.] finitely many polynomials
\[
P_1,\dots,P_J:\R^{d_x}\times \R^{d_y}\to \R,
\]
each of degree at most $4kr$;
\item[b.] an integer $\ell\in \{0,\dots,J\}$;
\item[c.] compact sets
\[
\Theta_+\subset \R^\ell,
\qquad
\Theta_-\subset \R^{J-\ell};
\]
\end{itemize}
such that
\[
\DGW_{r,k}(\mu,\nu)
=
\MGW_{r,k}(\mu,\nu)
+
4\inf_{v\in \Theta_-}\sup_{u\in \Theta_+}
\Bigl\{
-\|u\|^2+\|v\|^2+\OT_{c_{u,v}}(\mu,\nu)
\Bigr\},
\]
where
\[
c_{u,v}(x,y)
=
\sum_{i=1}^{\ell}u_i P_i(x,y)
-
\sum_{j=1}^{J-\ell}v_j P_{\ell+j}(x,y).
\]
If we write
\[
\Theta:=\Theta_+\times \Theta_-,
\qquad
c_\theta:=c_{u,v}\quad\text{for }\theta=(u,v)\in \Theta,
\]
then
\[
\CGW_{r,k}(\mu,\nu)
=
4\inf_{v\in \Theta_-}\sup_{u\in \Theta_+}
\Bigl\{
-\|u\|^2+\|v\|^2+\OT_{c_{u,v}}(\mu,\nu)
\Bigr\}.
\]
\end{theorem}

\begin{proof}
The full proof is given in Appendix~\ref{app:duality}.
\end{proof}

\begin{corollary}[Uniform compact sets under bounded support]\label{cor:uniform-boxes}
Assume in addition that
\[
\supp(\mu)\subset B_{d_x}(0,\rho_x),
\qquad
\supp(\nu)\subset B_{d_y}(0,\rho_y).
\]
Then the compact sets $\Theta_+$ and $\Theta_-$ in Theorem~\ref{thm:duality-main} may be chosen so that
they depend only on $r,k,d_x,d_y,\rho_x,\rho_y$, and not on the particular pair $(\mu,\nu)$.
\end{corollary}

\begin{proof}
See Appendix~\ref{app:duality}.
\end{proof}

\begin{remark}\label{rem:bridge}
Corollary~\ref{cor:uniform-boxes} is the bridge between generalized duality and the empirical-process
argument. It allows us to keep one compact parameter set $\Theta$ and one family of polynomial costs
$\{c_\theta:\theta\in \Theta\}$ fixed when we compare the centered population pair $(\mu,\nu)$ with the
centered empirical pair $(\widetilde\mu_n,\widetilde\nu_n)$.
\end{remark}

\begin{lemma}[Translation invariance and homogeneity]\label{lem:translation-homogeneity}
For $a\in \R^{d_x}$, $b\in \R^{d_y}$, and $\lambda>0$, define
\[
\tau_a(x):=x+a,
\qquad
\sigma_b(y):=y+b,
\qquad
S_\lambda(z):=\lambda z.
\]
Then, for all $\mu\in \cP(\R^{d_x})$ and $\nu\in \cP(\R^{d_y})$,
\[
\DGW_{r,k}\bigl((\tau_a)_\#\mu,(\sigma_b)_\#\nu\bigr)=\DGW_{r,k}(\mu,\nu)
\]
and
\[
\DGW_{r,k}\bigl((S_\lambda)_\#\mu,(S_\lambda)_\#\nu\bigr)
=
\lambda^{4kr}\DGW_{r,k}(\mu,\nu).
\]
The same identities hold for empirical measures obtained from translated or rescaled samples.
\end{lemma}
\begin{proof}
If $\pi\in \Pi(\mu,\nu)$, then $((\tau_a,\sigma_b))_\#\pi\in \Pi((\tau_a)_\#\mu,(\sigma_b)_\#\nu)$ and
\[
\|\tau_a(x)-\tau_a(x')\|=\|x-x'\|,
\qquad
\|\sigma_b(y)-\sigma_b(y')\|=\|y-y'\|.
\]
Therefore, the GW kernel is unchanged by these translations, and taking infima over couplings gives the
translation invariance.

\medskip

\noindent Similarly, $(S_\lambda,S_\lambda)_\#\pi\in \Pi((S_\lambda)_\#\mu,(S_\lambda)_\#\nu)$ and
\[
\|S_\lambda(x)-S_\lambda(x')\|^{2k}
=
\lambda^{2k}\|x-x'\|^{2k},
\qquad
\|S_\lambda(y)-S_\lambda(y')\|^{2k}
=
\lambda^{2k}\|y-y'\|^{2k}.
\]
Hence
\[
\bigl(
\|S_\lambda(x)-S_\lambda(x')\|^{2k}
-
\|S_\lambda(y)-S_\lambda(y')\|^{2k}
\bigr)^{2r}
=
\lambda^{4kr}
\bigl(
\|x-x'\|^{2k}
-
\|y-y'\|^{2k}
\bigr)^{2r}.
\]
Taking infima over couplings gives the homogeneity formula. The empirical-measure statement follows by
applying the same push-forward identities to the empirical measures of the translated or rescaled samples.
\end{proof}

\begin{remark}\label{rem:normalization}
In the proof of Theorem~\ref{thm:main} we use Lemma~\ref{lem:translation-homogeneity} with the dilation
$x\mapsto x/(2R)$. This geometric normalization is equivalent to renormalizing the polynomial basis
$(P_i)$ and the parameter boxes in Theorem~\ref{thm:duality-main} so that the induced cost family is
controlled on fixed unit-scale balls. We implement the normalization at the level of the measures because it
keeps the later entropy and regularity arguments cleaner.
\end{remark}

\section{Proof of the main theorem}\label{sec:proof-main}

Throughout this section, we assume without loss of generality that $d_x\le d_y$ and set $d_* :=d_x\wedge d_y=d_x$.

\medskip

\noindent The first step is to normalize the geometry. If $R=0$, then both $\mathcal X$ and $\mathcal Y$ are singletons, all empirical
measures are identical to the population measures, and the claim is trivial. We therefore assume $R>0$. Let
\[
m_\mu:=\int x\,d\mu(x),
\qquad
m_\nu:=\int y\,d\nu(y),
\]
and define the normalized measures
\[
\mu^\circ
:=
\left(x\mapsto \frac{x-m_\mu}{2R}\right)_\#\mu,
\qquad
\nu^\circ
:=
\left(y\mapsto \frac{y-m_\nu}{2R}\right)_\#\nu.
\]
If
\[
U_i:=\frac{X_i-m_\mu}{2R},
\qquad
V_j:=\frac{Y_j-m_\nu}{2R},
\]
then $(U_i)$ are i.i.d.\ with law $\mu^\circ$, $(V_j)$ are i.i.d.\ with law $\nu^\circ$, and their empirical
measures are
\[
\widehat\mu_n^\circ = \frac1n\sum_{i=1}^n \delta_{U_i},
\qquad
\widehat\nu_n^\circ = \frac1n\sum_{j=1}^n \delta_{V_j}.
\]
Because $\mathcal X\subset B_{d_x}(0,R)$ and $\mathcal Y\subset B_{d_y}(0,R)$, the centered-and-rescaled measures $\mu^\circ$
and $\nu^\circ$ are centered and satisfy
\[
\supp(\mu^\circ)\subset B_{d_x}(0,1),
\qquad
\supp(\nu^\circ)\subset B_{d_y}(0,1).
\]
By Lemma~\ref{lem:translation-homogeneity},
\[
\DGW_{r,k}(\mu,\nu)
=
(2R)^{4kr}\DGW_{r,k}(\mu^\circ,\nu^\circ),
\]
and likewise
\[
\DGW_{r,k}(\widehat\mu_n,\widehat\nu_n)
=
(2R)^{4kr}\DGW_{r,k}(\widehat\mu_n^\circ,\widehat\nu_n^\circ).
\]
Therefore
\begin{equation}\label{eq:rescaling-main}
\E\Bigl|
\DGW_{r,k}(\mu,\nu)
-
\DGW_{r,k}(\widehat\mu_n,\widehat\nu_n)
\Bigr|
=
(2R)^{4kr}
\E\Bigl|
\DGW_{r,k}(\mu^\circ,\nu^\circ)
-
\DGW_{r,k}(\widehat\mu_n^\circ,\widehat\nu_n^\circ)
\Bigr|.
\end{equation}
It is therefore enough to prove the theorem under the normalized hypotheses
\begin{equation}\label{eq:normalized-hypotheses}
\mu,\nu \text{ are centered},
\qquad
\supp(\mu)\subset B_{d_x}(0,1),
\qquad
\supp(\nu)\subset B_{d_y}(0,1).
\end{equation}
From now on we work under \eqref{eq:normalized-hypotheses} and suppress the superscript ${}^\circ$.

\subsection{Centering and decomposition}

Define the centered empirical measures by
\[
\widetilde\mu_n
=
\frac1n\sum_{i=1}^n \delta_{X_i-\overline X_n},
\qquad
\overline X_n:=\frac1n\sum_{i=1}^n X_i,
\]
and similarly
\[
\widetilde\nu_n
=
\frac1n\sum_{j=1}^n \delta_{Y_j-\overline Y_n},
\qquad
\overline Y_n:=\frac1n\sum_{j=1}^n Y_j.
\]
Since $\supp(\mu)\subset B_{d_x}(0,1)$ and $\supp(\nu)\subset B_{d_y}(0,1)$, we have
\[
\supp(\widetilde\mu_n)\subset B_{d_x}(0,2),
\qquad
\supp(\widetilde\nu_n)\subset B_{d_y}(0,2)
\]
almost surely.

\begin{proposition}\label{prop:centering-decomp}
Under the normalized hypotheses \eqref{eq:normalized-hypotheses},
\[
\begin{aligned}
\Bigl|
\DGW_{r,k}(\mu,\nu)
-
\DGW_{r,k}(\widehat\mu_n,\widehat\nu_n)
\Bigr|
&\le
\Bigl|
\MGW_{r,k}(\mu,\nu)-\MGW_{r,k}(\widehat\mu_n,\widehat\nu_n)
\Bigr|
\\
&\quad+
\Bigl|
\MGW_{r,k}(\widehat\mu_n,\widehat\nu_n)-\MGW_{r,k}(\widetilde\mu_n,\widetilde\nu_n)
\Bigr|
\\
&\quad+
\Bigl|
\CGW_{r,k}(\mu,\nu)-\CGW_{r,k}(\widetilde\mu_n,\widetilde\nu_n)
\Bigr|.
\end{aligned}
\]
\end{proposition}

\begin{proof}
Translation invariance from Lemma~\ref{lem:translation-homogeneity} gives
\[
\DGW_{r,k}(\widehat\mu_n,\widehat\nu_n)
=
\DGW_{r,k}(\widetilde\mu_n,\widetilde\nu_n).
\]
Applying the decomposition \eqref{eq:marginal-coupling-decomp} to $(\mu,\nu)$ and $(\widetilde\mu_n,\widetilde\nu_n)$, and
then adding and subtracting $\MGW_{r,k}(\widehat\mu_n,\widehat\nu_n)$, gives the claim.
\end{proof}

\subsection{Marginal term}

\begin{proposition}\label{prop:marginal-structure}
For every $r,k\ge 1$, $\MGW_{r,k}(\mu,\nu)$ is a polynomial in marginal moments of $\mu$ and $\nu$ of order
at most $4kr$. Each monomial is a product of at most two moments, and all coefficients depend only on
$r,k,d_x,d_y$.
\end{proposition}
\begin{proof}
See Appendix~\ref{app:marginal}.
\end{proof}

\bigskip

\begin{proposition}\label{prop:marginal-pop-vs-emp}
Under the normalized hypotheses \eqref{eq:normalized-hypotheses},
\[
\E\Bigl|
\MGW_{r,k}(\mu,\nu)-\MGW_{r,k}(\widehat\mu_n,\widehat\nu_n)
\Bigr|
\lesssim_{r,k,d_x,d_y} n^{-1/2}
\]
\end{proposition}

\begin{proof}
By Proposition~\ref{prop:marginal-structure}, $\MGW_{r,k}(\mu,\nu)$ is a finite linear combination of monomials of
the form
\[
M^X_\alpha(\mu),
\qquad
M^X_\alpha(\mu)M^X_\beta(\mu),
\qquad
M^X_\alpha(\mu)M^Y_\gamma(\nu),
\]
and the analogous expressions involving only moments of $\nu$, where $\alpha, \beta, \gamma$ are multi-indices of the appropriate dimensions. Thus, it suffices to estimate the empirical error for one such monomial.

\medskip

\noindent \textbf{Step 1.} For every multi-index $\alpha$ with $|\alpha|\le 4kr$, the support assumption $\supp (\mu) \subset B_{d_x}(0,1)$ implies
\[
|x^\alpha| = \prod_{i=1}^{d_x} |x_i|^{\alpha_i}
\le
\|x\|^{|\alpha|}
\le
1
\qquad
\text{for all }x\in \supp(\mu).
\]
In particular, this implies that $\E |Z^\alpha|^2 \leq 1$ for any random variable $Z \sim \mu$ and any multi-index $\alpha$. Set
\[
Y := M^X_\alpha(\widehat\mu_n)-M^X_\alpha(\mu)
=
\frac1n\sum_{i=1}^n
\bigl(
X_i^\alpha-\E[X_1^\alpha]
\bigr),
\]
Note that \(\mathbb E Y=0\) and \(\Var(Y) = \Var(M^X_\alpha(\widehat\mu_n))\). Using independence, 
\[
\Var\bigl(M^X_\alpha(\widehat\mu_n)\bigr)
=
\frac1n\Var(X_1^\alpha)
\le
\frac1n\E|X_1^\alpha|^2
\le
\frac1n.
\]
Hence, by Cauchy--Schwarz,
\[
\mathbb E|Y|
\le (\mathbb E Y^2)^{1/2}
= \sqrt{\operatorname{Var}(Y)}
= \sqrt{\operatorname{Var}(M^X_\alpha(\widehat\mu_n))}
\le \frac1{\sqrt n}.
\]
Thus, we conclude that
\[
\E\bigl|M^X_\alpha(\widehat\mu_n)-M^X_\alpha(\mu)\bigr|
\le
\sqrt{\Var(M^X_\alpha(\widehat\mu_n))}
\le
\frac1{\sqrt n}.
\]
The same estimate holds for moments of $\nu$:
\[
\E\bigl|M^Y_\alpha(\widehat\nu_n)-M^Y_\alpha(\nu)\bigr|
\le
\sqrt{\Var(M^Y_\alpha(\widehat\nu_n))}
\le
\frac1{\sqrt n}.
\]

\medskip

\noindent \textbf{Step 2.} We now estimate the mixed case $U^X_{\alpha,\beta}(\mu)=M^X_\alpha(\mu)M^X_\beta(\mu)$ as follows. Observe that:
\[
\bigl|
M^X_\alpha(\mu)M^X_\beta(\mu)-M^X_\alpha(\widehat\mu_n)M^X_\beta(\widehat\mu_n)
\bigr|
\le
|M^X_\alpha(\mu)-M^X_\alpha(\widehat\mu_n)|\,|M^X_\beta(\mu)|
+
|M^X_\alpha(\widehat\mu_n)|\,|M^X_\beta(\mu)-M^X_\beta(\widehat\mu_n)|.
\]
Since
\[
|M^X_\beta(\mu)|\le 1,
\qquad
|M^X_\alpha(\widehat\mu_n)|\le \frac1n\sum_{i=1}^n |X_i^\alpha|\le 1,
\]
taking expectations gives
\[
\E\bigl|U^X_{\alpha,\beta}(\mu)-U^X_{\alpha,\beta}(\widehat\mu_n)\bigr|
\le
\frac{2}{\sqrt n}.
\]

\medskip

\noindent \textbf{Step 3.} For a mixed monomial \(U_{\alpha, \gamma}(\mu,\nu)=M^X_\alpha(\mu)M^Y_\gamma(\nu)\), we similarly have
\[
|U_{\alpha, \gamma}(\mu,\nu)-U_{\alpha, \gamma}(\widehat\mu_n,\widehat\nu_n)|
\le
|M^X_\alpha(\mu)-M^X_\alpha(\widehat\mu_n)|\,|M^Y_\gamma(\nu)|
+
|M^X_\alpha(\widehat\mu_n)|\,|M^Y_\gamma(\nu)-M^Y_\gamma(\widehat\nu_n)|.
\]
Using \(|M^Y_\gamma(\nu)|\le 1\), \(|M^X_\alpha(\widehat\mu_n)|\le 1\), and the one-moment bounds above gives
\[
\mathbb E|U_{\alpha, \gamma}(\mu,\nu)-U_{\alpha, \gamma}(\widehat\mu_n,\widehat\nu_n)|\le \frac{2}{\sqrt n}.
\]
Therefore, summing over the finitely many
monomials in $\MGW_{r,k}$, we obtain
\[
\E\Bigl|
\MGW_{r,k}(\mu,\nu)-\MGW_{r,k}(\widehat\mu_n,\widehat\nu_n)
\Bigr|
\lesssim_{r,k,d_x,d_y} n^{-1/2}
\]
\end{proof}

\begin{proposition}\label{prop:marginal-emp-vs-centered}
Under the normalized hypotheses \eqref{eq:normalized-hypotheses},
\[
\E\Bigl|
\MGW_{r,k}(\widehat\mu_n,\widehat\nu_n)-\MGW_{r,k}(\widetilde\mu_n,\widetilde\nu_n)
\Bigr| \lesssim_{r,k,d_x,d_y} n^{-1/2}
\]
\end{proposition}

\begin{proof}
The proof is similar to the above. See Appendix~\ref{app:marginal}.
\end{proof}

\bigskip

\noindent Combining Propositions~\ref{prop:marginal-pop-vs-emp} and \ref{prop:marginal-emp-vs-centered} gives
\begin{equation}\label{eq:marginal-final}
\begin{aligned}
\E\Bigl[
\Bigl|
\MGW_{r,k}(\mu,\nu)-\MGW_{r,k}(\widehat\mu_n,\widehat\nu_n)
\Bigr|
&+
\Bigl|
\MGW_{r,k}(\widehat\mu_n,\widehat\nu_n)-\MGW_{r,k}(\widetilde\mu_n,\widetilde\nu_n)
\Bigr|
\Bigr]
 \lesssim_{r,k,d_x,d_y} n^{-1/2}
\end{aligned}
\end{equation}

\bigskip

\subsection{Reduction of the transport term to a fixed cost family}

Set
\[
B_x:=B_{d_x}(0,2),
\qquad
B_y:=B_{d_y}(0,2).
\]
By Corollary~\ref{cor:uniform-boxes}, applied with $(\rho_x,\rho_y)=(2,2)$, there exist compact sets
\[
\Theta_+,
\qquad
\Theta_-,
\]
depending only on $r,k,d_x,d_y$, such that the generalized duality representation from
Theorem~\ref{thm:duality-main} holds for the centered population pair $(\mu,\nu)$ and the
centered empirical pair $(\widetilde\mu_n,\widetilde\nu_n)$. We therefore set
\[
\Theta:=\Theta_+\times \Theta_-,
\qquad
c_\theta:=c_{u,v}\quad\text{for }\theta=(u,v)\in \Theta,
\]
and keep this compact parameter set and cost family fixed throughout the rest of the proof.

The transport term has therefore been reduced to a uniform empirical-process problem for the fixed family
\[
\{\OT_{c_\theta}:\theta\in \Theta\}.
\]
From this point onward, the strategy is to isolate two ambient classes that contain all optimal first and second
dual potentials, and then to estimate the entropy of those classes separately. The first class is controlled
by semiconcavity. The second class is obtained from the first one by taking $c_\theta$-transforms and
varying $\theta$ over the compact set $\Theta$.

\subsection{Regularity of the dual potential classes}

The next lemmas form the regularity package for optimal transport under smooth or semiconcave costs;
compare \cite{ZhangEtAl,HundrieserStaudtMunk,ManoleNilesWeed}. We first record the regularity properties
of the polynomial cost family $\{c_\theta\}_{\theta\in\Theta}$.

\begin{lemma}\label{lem:cost-semiconcavity}
There exists a constant
\[
\Lambda=\Lambda(r,k,d_x,d_y)<\infty
\]
such that, for every $\theta\in \Theta$ and every $y\in B_y$, the map
\[
x\longmapsto c_\theta(x,y)
\]
is $\Lambda$-semiconcave on $B_x$. Likewise, for every $\theta\in \Theta$ and every $x\in B_x$, the map
\[
y\longmapsto c_\theta(x,y)
\]
is $\Lambda$-semiconcave on $B_y$.
\end{lemma}

\begin{proof}
Each $c_\theta$ is a polynomial on $\R^{d_x}\times \R^{d_y}$ of degree at most $4kr$, with coefficients
depending continuously on $\theta\in \Theta$. Hence every entry of the Hessians $\nabla_{xx}^2 c_\theta$ and
$\nabla_{yy}^2 c_\theta$ is a polynomial of degree at most $4kr-2$ whose coefficients are uniformly bounded
over $\Theta$. Since $\Theta\times B_x\times B_y$ is compact, there exists $\Lambda<\infty$ such that
\[
\nabla_{xx}^2 c_\theta(x,y) \preceq \Lambda I,
\qquad
\nabla_{yy}^2 c_\theta(x,y) \preceq \Lambda I
\]
for all $(\theta,x,y)\in \Theta\times B_x\times B_y$. Therefore, the lemma holds. 
\end{proof}

\bigskip

\noindent For later use define
\[
L_x
:=
\sup_{\theta\in \Theta}
\sup_{(x,y)\in B_x\times B_y}
\|\nabla_x c_\theta(x,y)\|,
\qquad
L_y
:=
\sup_{\theta\in \Theta}
\sup_{(x,y)\in B_x\times B_y}
\|\nabla_y c_\theta(x,y)\|,
\]
and
\[
L_{\mathrm{pot}}:=L_x\vee L_y,
\qquad
C_{\mathrm{cost}}:=
\sup_{\theta\in \Theta}
\sup_{(x,y)\in B_x\times B_y}
|c_\theta(x,y)|.
\]
These quantities are finite because $\Theta\times B_x\times B_y$ is compact and the cost family is
continuous.

\bigskip

\begin{lemma}\label{lem:dual-pair-bounds}
Fix $\theta\in \Theta$, and let $\mu',\nu'$ be probability measures supported on compact sets
\[
\mathcal{X}'\subset B_x,
\qquad
\mathcal{Y}'\subset B_y.
\]
Then there exists an optimal pair $(\phi,\psi)$ for $\OT_{c_\theta}(\mu',\nu')$ such that
\[
\phi(x)=\inf_{y\in \mathcal{Y}'}\bigl(c_\theta(x,y)-\psi(y)\bigr),
\qquad
\psi(y)=\inf_{x\in \mathcal{X}'}\bigl(c_\theta(x,y)-\phi(x)\bigr).
\]
Moreover, after normalizing by an additive constant, we may arrange
\[
\|\phi\|_{\Lip(\mathcal{X}')}\le L_{\mathrm{pot}},
\qquad
\|\psi\|_{\Lip(\mathcal{Y}')}\le L_{\mathrm{pot}},
\]
and
\[
\|\phi\|_{L^\infty(\mathcal{X}')}\le C_{\varphi},
\qquad
\|\psi\|_{L^\infty(\mathcal{Y}')}\le C_{\varphi},
\]
where
\[
C_{\varphi}:=C_{\mathrm{cost}}+4L_{\mathrm{pot}}.
\]
\end{lemma}

\begin{proof}
Kantorovich duality for the continuous cost $c_\theta$ on the compact set $\mathcal{X}'\times \mathcal{Y}'$ gives an optimal
dual pair $(u,v)$; see, for example, \cite{Villani}. Replacing this pair by its $c$-transforms gives another
optimal pair $(\phi,\psi)$ satisfying the displayed identities. The Lipschitz estimates follow from
these formulae:
\[
|\phi(x_1)-\phi(x_2)|
\le
\sup_{y\in Y'}|c_\theta(x_1,y)-c_\theta(x_2,y)|
\le
L_x\|x_1-x_2\|,
\]
and similarly for $\psi$.

\medskip

\noindent For the $L^\infty$ bounds, fix $x_0\in \mathcal{X}'$ and replace $(\phi,\psi)$ by
\[
(\phi-\phi(x_0),\ \psi+\phi(x_0));
\]
this preserves optimality and the $c$-transform identities. Then $\phi(x_0)=0$, so
\[
|\phi(x)|
\le
L_{\mathrm{pot}}\,\diam(B_x)
=
4L_{\mathrm{pot}}
\qquad
\text{for all }x\in \mathcal{X}',
\]
which implies $\|\phi\|_{L^\infty(\mathcal{X}')}\le 4L_{\mathrm{pot}}\le C_{\varphi}$. Since
\[
\psi(y)=\inf_{x\in X'}\bigl(c_\theta(x,y)-\phi(x)\bigr),
\]
we get
\[
-C_{\mathrm{cost}}-4L_{\mathrm{pot}}
\le
\psi(y)
\le
C_{\mathrm{cost}}
\qquad
\text{for all }y\in \mathcal{Y}',
\]
hence $\|\psi\|_{L^\infty(\mathcal{Y}')}\le C_{\mathrm{cost}}+4L_{\mathrm{pot}}=C_{\varphi}$.
\end{proof}

\bigskip

\begin{lemma}\label{lem:first-potential-extension}
Fix $\theta\in \Theta$, and let $(\phi,\psi)$ be an optimal pair as in
Lemma~\ref{lem:dual-pair-bounds}. Define
\[
\overline\phi(x):=\inf_{y\in \mathcal{Y}'}\bigl(c_\theta(x,y)-\psi(y)\bigr),
\qquad
x\in B_x,
\]
and
\[
\overline\psi(y):=\inf_{x\in \mathcal{X}'}\bigl(c_\theta(x,y)-\phi(x)\bigr),
\qquad
y\in B_y.
\]
Then:
\begin{itemize}[leftmargin=2em]
\item[a.] $\overline\phi|_{\mathcal{X}'}=\phi$ and $\overline\psi|_{\mathcal{Y}'}=\psi$;
\item[b.] $\overline\phi$ is $\Lambda$-semiconcave on $B_x$, and $\overline\psi$ is $\Lambda$-semiconcave on
$B_y$;
\item[c.] $\|\overline\phi\|_{\Lip(B_x)}\le L_{\mathrm{pot}}$ and $\|\overline\psi\|_{\Lip(B_y)}\le L_{\mathrm{pot}}$;
\item[d.] $\|\overline\phi\|_{L^\infty(B_x)}\le C_{\mathrm{ext}}$ and $\|\overline\psi\|_{L^\infty(B_y)}\le C_{\mathrm{ext}}$,
where
\[
C_{\mathrm{ext}}:=C_{\mathrm{cost}}+C_{\varphi}.
\]
\end{itemize}
\end{lemma}
\begin{proof}
\textbf{(a):} For \(x\in \mathcal{X}'\), Lemma~\ref{lem:dual-pair-bounds} gives
\[
\phi(x)=\inf_{y\in \mathcal{Y}'}\bigl(c_\theta(x,y)-\psi(y)\bigr),
\]
so by definition of \(\overline\phi\),
\[
\overline\phi(x)=\phi(x), \qquad x\in \mathcal{X}'.
\]
Thus, \(\overline\phi|_{X'}=\phi\). Now let \(y\in \mathcal{Y}'\). Since
\[
\overline\phi(x)=\inf_{y'\in \mathcal{Y}'}\bigl(c_\theta(x,y')-\psi(y')\bigr),
\qquad x\in B_x,
\]
we have, for every \(x\in B_x\),
\[
\overline\phi(x)\le c_\theta(x,y)-\psi(y).
\]
Rearranging gives
\[
c_\theta(x,y)-\overline\phi(x)\ge \psi(y),
\qquad x\in B_x.
\]
Taking the infimum over \(x\in \mathcal{X}'\) gives
\[
\overline\psi(y)
=
\inf_{x\in \mathcal{X}'}\bigl(c_\theta(x,y)-\overline\phi(x)\bigr)
\ge \psi(y).
\]
On the other hand, since \(\overline\phi=\phi\) on \(\mathcal{X}'\),
\[
\overline\psi(y)
=
\inf_{x\in \mathcal{X}'}\bigl(c_\theta(x,y)-\overline\phi(x)\bigr)
=
\inf_{x\in \mathcal{X}'}\bigl(c_\theta(x,y)-\phi(x)\bigr)
=
\psi(y),
\]
where the last equality again follows from Lemma~\ref{lem:dual-pair-bounds}. Hence
\[
\overline\psi(y)=\psi(y), \qquad y\in \mathcal{Y}',
\]
so \(\overline\psi|_{\mathcal{Y}'}=\psi\).

\medskip

\noindent \textbf{(b)}: Fix \(y\in \mathcal{Y}'\). By Lemma~\ref{lem:cost-semiconcavity}, the map
\[
x\longmapsto c_\theta(x,y)-\frac{\Lambda}{2}\|x\|^2
\]
is concave on \(B_x\). Therefore
\[
x\longmapsto \frac{\Lambda}{2}\|x\|^2-c_\theta(x,y)+\psi(y)
\]
is convex on \(B_x\). Since
\[
\frac{\Lambda}{2}\|x\|^2-\overline\phi(x)
=
\sup_{y\in \mathcal{Y}'}
\left[
\frac{\Lambda}{2}\|x\|^2-c_\theta(x,y)+\psi(y)
\right],
\]
the left-hand side is a supremum of convex functions and is therefore convex. Thus
\(\overline\phi\) is \(\Lambda\)-semiconcave on \(B_x\). The argument for \(\overline\psi\) on \(B_y\) is identical. 

\medskip

\noindent \textbf{(c)}: The Lipschitz bounds are proved as as in Lemma~\ref{lem:dual-pair-bounds}. For
\(x_1,x_2\in B_x\),
\[
|\overline\phi(x_1)-\overline\phi(x_2)|
\le
\sup_{y\in \mathcal{Y}'}|c_\theta(x_1,y)-c_\theta(x_2,y)|
\le
L_x\|x_1-x_2\|
\le
L_{\mathrm{pot}}\|x_1-x_2\|,
\]
and similarly \(\|\overline\psi\|_{\Lip(B_y)}\le L_{\mathrm{pot}}\).

\medskip

\noindent \textbf{(d)}:  Finally, for every \(x\in B_x\) and \(y\in \mathcal{Y}'\), Lemma~\ref{lem:dual-pair-bounds} gives
\[
|\psi(y)|\le C_{\varphi}
\]
and the definition of \(C_{\mathrm{cost}}\) gives
\[
|c_\theta(x,y)|\le C_{\mathrm{cost}}.
\]
Hence
\[
-C_{\mathrm{cost}}-C_{\varphi}\le c_\theta(x,y)-\psi(y)\le C_{\mathrm{cost}}+C_{\varphi}.
\]
Taking the infimum over \(y\in \mathcal{Y}'\),
\[
|\overline\phi(x)|\le C_{\mathrm{cost}}+C_{\varphi}, \qquad x\in B_x.
\]
Thus
\[
\|\overline\phi\|_{L^\infty(B_x)}\le C_{\mathrm{ext}},
\qquad
C_{\mathrm{ext}}:=C_{\mathrm{cost}}+C_{\varphi}.
\]
The same argument, using \(|\phi(x)|\le C_{\varphi}\) on \(\mathcal{X}'\), gives
\[
\|\overline\psi\|_{L^\infty(B_y)}\le C_{\mathrm{ext}}.
\]
\end{proof}

\bigskip

\begin{lemma}\label{lem:global-feasible-second}
Fix $\theta\in \Theta$. Let $\mu',\nu'$ be probability measures supported on compact sets
\[
\mathcal{X}'\subset B_x,
\qquad
\mathcal{Y}'\subset B_y,
\]
and let $(\phi,\psi)$ be an optimal pair for $\OT_{c_\theta}(\mu',\nu')$, normalized as in
Lemma~\ref{lem:dual-pair-bounds}. Let $\overline\phi$ be the extension from
Lemma~\ref{lem:first-potential-extension}. Define
\[
\mathsf C_\theta f(y):=\inf_{x\in B_x}\bigl(c_\theta(x,y)-f(x)\bigr),
\qquad
y\in B_y,
\]
and set
\[
\widetilde\psi:=\mathsf C_\theta\overline\phi.
\]
Then:
\begin{itemize}[leftmargin=2em]
\item[a.] $\widetilde\psi(y)=\psi(y)$ for all $y\in \mathcal{Y}'$;
\item[b.] $(\overline\phi,\widetilde\psi)$ is dual feasible on $B_x\times B_y$, i.e.
\[
\overline\phi(x)+\widetilde\psi(y)\le c_\theta(x,y)
\qquad
\text{for all }(x,y)\in B_x\times B_y;
\] and 
\[
\OT_{c_\theta}(\mu',\nu')
=
\int \overline\phi\,d\mu' + \int \widetilde\psi\,d\nu'.
\]
\end{itemize}
\end{lemma}

\begin{proof}
\textbf{(a):} For $y\in \mathcal{Y}'$ and $x\in B_x$, the definition of $\overline\phi$ gives
\[
\overline\phi(x)\le c_\theta(x,y)-\psi(y),
\]
hence
\[
c_\theta(x,y)-\overline\phi(x)\ge \psi(y).
\]
Taking the infimum over $x\in B_x$,
\[
\mathsf C_\theta\overline\phi(y)\ge \psi(y).
\]
On the other hand, since $\overline\phi=\phi$ on $\mathcal{X}'$ by
Lemma~\ref{lem:first-potential-extension},
\[
\mathsf C_\theta\overline\phi(y)
\le
\inf_{x\in \mathcal{X}'}\bigl(c_\theta(x,y)-\overline\phi(x)\bigr)
=
\inf_{x\in \mathcal{X}'}\bigl(c_\theta(x,y)-\phi(x)\bigr)
=
\psi(y).
\]
Thus $\widetilde\psi(y)=\mathsf C_\theta\overline\phi(y)=\psi(y)$ for all $y\in \mathcal{Y}'$.

\medskip

\noindent \textbf{(b):} By definition of the $c_\theta$-transform,
\[
\widetilde\psi(y)=\inf_{x\in B_x}\bigl(c_\theta(x,y)-\overline\phi(x)\bigr),
\]
so
\[
\overline\phi(x)+\widetilde\psi(y)\le c_\theta(x,y)
\qquad
\text{for all }(x,y)\in B_x\times B_y.
\]
Hence $(\overline\phi,\widetilde\psi)$ is dual feasible on $B_x\times B_y$. Finally, because $\overline\phi=\phi$ on $\mathcal{X}'$ and $\widetilde\psi=\psi$ on $\mathcal{Y}'$,
\[
\int \overline\phi\,d\mu' + \int \widetilde\psi\,d\nu'
=
\int \phi\,d\mu' + \int \psi\,d\nu'
=
\OT_{c_\theta}(\mu',\nu').
\]
\end{proof}

\bigskip

\noindent We now introduce the two potential classes that drive the empirical-process argument:
\[
\cF_X
:=
\Bigl\{
f:B_x\to \R:
f \text{ is }\Lambda\text{-semiconcave on }B_x,
\ \|f\|_{\Lip(B_x)}\le L_{\mathrm{pot}},
\ \|f\|_{L^\infty(B_x)}\le C_{\mathrm{ext}}
\Bigr\},
\]
and
\[
\cG_Y
:=
\bigcup_{\theta\in \Theta} \mathsf C_\theta(\cF_X)
\subset C(B_y).
\]
By Lemmas~\ref{lem:first-potential-extension} and \ref{lem:global-feasible-second}, every optimal first
potential extends to an element of $\cF_X$, and every optimal dual pair may be represented by
functions from the ambient classes $\cF_X$ and $\cG_Y$ without changing its value on the original
marginals.

\medskip

\noindent We will also use that every $g\in \cG_Y$ is $L_y$-Lipschitz on $B_y$. Indeed, if $g=\mathsf C_\theta f$ with
$\theta\in \Theta$ and $f\in \cF_X$, then for any $y_1,y_2\in B_y$,
\[
g(y_1)-g(y_2)
\le
\sup_{x\in B_x}\bigl(c_\theta(x,y_1)-c_\theta(x,y_2)\bigr)
\le
L_y\|y_1-y_2\|,
\]
and the same inequality with $y_1$ and $y_2$ interchanged gives
\[
|g(y_1)-g(y_2)|\le L_y\|y_1-y_2\|\le L_{\mathrm{pot}}\|y_1-y_2\|.
\]

\begin{proposition}\label{prop:transport-reduction}
There exists a constant $C=C(r,k,d_x,d_y)$ such that
\[
\E\Bigl|
\CGW_{r,k}(\mu,\nu)-\CGW_{r,k}(\widetilde\mu_n,\widetilde\nu_n)
\Bigr|
\le
4\,\E\sup_{f\in \cF_X}|\Pair{\mu-\widehat\mu_n}{f}|
+
4\,\E\sup_{g\in \cG_Y}|\Pair{\nu-\widehat\nu_n}{g}|
+
\frac{C}{\sqrt n}.
\]
\end{proposition}

\begin{proof}
By Theorem~\ref{thm:duality-main} and the inequality
\[
\left|
\inf_{v\in \Theta_-}\sup_{u\in \Theta_+} A_{u,v}
-
\inf_{v\in \Theta_-}\sup_{u\in \Theta_+} B_{u,v}
\right|
\le
\sup_{(u,v)\in \Theta_+\times \Theta_-}|A_{u,v}-B_{u,v}|,
\]
we have
\[
\Bigl|
\CGW_{r,k}(\mu,\nu)-\CGW_{r,k}(\widetilde\mu_n,\widetilde\nu_n)
\Bigr|
\le
4\sup_{\theta\in \Theta}
\Bigl|
\OT_{c_\theta}(\mu,\nu)-\OT_{c_\theta}(\widetilde\mu_n,\widetilde\nu_n)
\Bigr|.
\]
Set
\[
\Delta_n
:=
\sup_{f\in \cF_X}|\Pair{\mu-\widetilde\mu_n}{f}|
+
\sup_{g\in \cG_Y}|\Pair{\nu-\widetilde\nu_n}{g}|.
\]
Fix $\theta\in \Theta$. Let $(\phi_\theta,\psi_\theta)$ be an optimal pair for $\OT_{c_\theta}(\mu,\nu)$,
normalized as in Lemma~\ref{lem:dual-pair-bounds}. Let $\overline\phi_\theta$ be the extension from
Lemma~\ref{lem:first-potential-extension}, and define
\[
\widetilde\psi_\theta:=\mathsf C_\theta\overline\phi_\theta.
\]
By Lemma~\ref{lem:global-feasible-second},
\[
\OT_{c_\theta}(\mu,\nu)
=
\int \overline\phi_\theta\,d\mu + \int \widetilde\psi_\theta\,d\nu,
\]
and $(\overline\phi_\theta,\widetilde\psi_\theta)$ is dual feasible on $B_x\times B_y$. Since
$\widetilde\mu_n$ and $\widetilde\nu_n$ are supported on $B_x$ and $B_y$, dual feasibility gives
\[
\OT_{c_\theta}(\widetilde\mu_n,\widetilde\nu_n)
\ge
\int \overline\phi_\theta\,d\widetilde\mu_n
+
\int \widetilde\psi_\theta\,d\widetilde\nu_n.
\]
Therefore
\[
\OT_{c_\theta}(\mu,\nu)-\OT_{c_\theta}(\widetilde\mu_n,\widetilde\nu_n)
\le
\Pair{\mu-\widetilde\mu_n}{\overline\phi_\theta} + \Pair{\nu-\widetilde\nu_n}{\widetilde\psi_\theta}
\le
\Delta_n.
\]
Applying the same argument with an optimal pair for $\OT_{c_\theta}(\widetilde\mu_n,\widetilde\nu_n)$ gives
\[
\OT_{c_\theta}(\widetilde\mu_n,\widetilde\nu_n)-\OT_{c_\theta}(\mu,\nu)
\le
\Delta_n.
\]
Hence
\[
\Bigl|
\OT_{c_\theta}(\mu,\nu)-\OT_{c_\theta}(\widetilde\mu_n,\widetilde\nu_n)
\Bigr|
\le
\Delta_n.
\]
Taking the supremum over $\theta\in \Theta$, we obtain
\[
\Bigl|
\CGW_{r,k}(\mu,\nu)-\CGW_{r,k}(\widetilde\mu_n,\widetilde\nu_n)
\Bigr|
\le
4\Delta_n.
\]
Now fix $f\in \cF_X$. Then
\[
\Pair{\mu-\widetilde\mu_n}{f}
=
\Pair{\mu-\widehat\mu_n}{f}
+
\frac1n\sum_{i=1}^n\bigl(f(X_i)-f(X_i-\overline X_n)\bigr).
\]
Since every $f\in \cF_X$ is $L_{\mathrm{pot}}$-Lipschitz on $B_x$,
\[
|\Pair{\mu-\widetilde\mu_n}{f}|
\le
|\Pair{\mu-\widehat\mu_n}{f}|
+
L_{\mathrm{pot}}\|\overline X_n\|.
\]
Taking the supremum over $f\in \cF_X$ gives
\[
\sup_{f\in \cF_X}|\Pair{\mu-\widetilde\mu_n}{f}|
\le
\sup_{f\in \cF_X}|\Pair{\mu-\widehat\mu_n}{f}|
+
L_{\mathrm{pot}}\|\overline X_n\|.
\]

\medskip

\noindent The same argument on the $\nu$-side, using the Lipschitz bound just established for every
$g\in \cG_Y$, gives
\[
\sup_{g\in \cG_Y}|\Pair{\nu-\widetilde\nu_n}{g}|
\le
\sup_{g\in \cG_Y}|\Pair{\nu-\widehat\nu_n}{g}|
+
L_{\mathrm{pot}}\|\overline Y_n\|.
\]
Therefore
\[
\begin{aligned}
\Bigl|
\CGW_{r,k}(\mu,\nu)-\CGW_{r,k}(\widetilde\mu_n,\widetilde\nu_n)
\Bigr|
&\le
4\sup_{f\in \cF_X}|\Pair{\mu-\widehat\mu_n}{f}|
+
4\sup_{g\in \cG_Y}|\Pair{\nu-\widehat\nu_n}{g}|
\\
&\quad+
4L_{\mathrm{pot}}\|\overline X_n\|
+
4L_{\mathrm{pot}}\|\overline Y_n\|.
\end{aligned}
\]
Taking expectations and using Cauchy--Schwarz,
\[
\mathbb E\|\overline X_n\|\le \bigl(\mathbb E\|\overline X_n\|^2\bigr)^{1/2},
\qquad
\mathbb E\|\overline Y_n\|\le \bigl(\mathbb E\|\overline Y_n\|^2\bigr)^{1/2}.
\]
Moreover,
\[
\mathbb E\|\overline X_n\|^2
=
\frac1{n^2}\sum_{i=1}^n \mathbb E\|X_i\|^2
+
\frac1{n^2}\sum_{i\neq j}\mathbb E[X_i]\cdot \mathbb E[X_j].
\]
Under the normalized hypotheses \eqref{eq:normalized-hypotheses}, \(\mathbb E[X_i]=0\), so
\[
\mathbb E\|\overline X_n\|^2=\frac1n\,\mathbb E\|X_1\|^2\le \frac1n.
\]
Similarly, \(\mathbb E\|\overline Y_n\|^2\le 1/n\). Therefore
\[
\mathbb E\|\overline X_n\|+\mathbb E\|\overline Y_n\|\le \frac{2}{\sqrt n},
\]
The result follows after absorbing the factor \(8L_{\mathrm{pot}}\) into the constant \(C\).
\end{proof}

\subsection{Entropy bounds and completion of the transport estimate}

By Lemma~\ref{lem:semicolon-entropy} and Proposition~\ref{prop:g-entropy} below, the classes
\(\cF_X\) and \(\cG_Y\) have finite covering numbers in the sup norm for every sufficiently small
\(\varepsilon>0\). Since both classes are also uniformly bounded in the sup norm, they are totally
bounded, and hence separable. Let
\(\mathcal{D}_1 \subset \cF_X\) and \(\mathcal{D}_2 \subset \cG_Y\) be countable dense subsets. For any signed measure
\(\lambda\) with total variation at most \(2\), the map \(f\mapsto |\Pair{\lambda}{f}|\) is \(2\)-Lipschitz with
respect to \(\|\cdot\|_\infty\), since
\[
\bigl||\Pair{\lambda}{f}|-|\Pair{\lambda}{h}|\bigr|
\le |\Pair{\lambda}{f-h}|
\le \|\lambda\|_{\mathrm{TV}}\|f-h\|_\infty
\le 2\|f-h\|_\infty.
\]
Applying this with \(\lambda=\mu-\widehat\mu_n\) and \(\lambda=\nu-\widehat\nu_n\), we get
\[
\sup_{f\in \cF_X}|\Pair{\mu-\widehat\mu_n}{f}|
=
\sup_{f\in \mathcal{D}_1}|\Pair{\mu-\widehat\mu_n}{f}|,
\qquad
\sup_{g\in \cG_Y}|\Pair{\nu-\widehat\nu_n}{g}|
=
\sup_{g\in \mathcal{D}_2}|\Pair{\nu-\widehat\nu_n}{g}|
\]
pointwise. In particular, the empirical-process suprema appearing below are measurable, so ordinary expectations may be used without further comment.

\begin{proposition}\label{prop:f-empirical}
There exists a constant $C=C(r,k,d_x,d_y)$ such that
\[
\E\sup_{f\in \cF_X}|\Pair{\mu-\widehat\mu_n}{f}|
\le
C\,
\rho_n(d_*).
\]
\end{proposition}

\begin{proof}
By Lemma~\ref{lem:semicolon-entropy}, with $D=B_x$, there exists $C_0$ such that
\[
\log \Cov\bigl(\varepsilon,\cF_X,\|\cdot\|_\infty\bigr)\le C_0\varepsilon^{-d_*/2}.
\]
Dudley's entropy integral inequality for uniformly bounded classes
\cite{Dudley,vdVW} therefore gives
\[
\E\sup_{f\in \cF_X}|\Pair{\mu-\widehat\mu_n}{f}|
\le
\inf_{\alpha>0}
\left\{
\alpha
+
\frac{C}{\sqrt n}\int_\alpha^{2C_{\mathrm{ext}}}\varepsilon^{-d_*/4}\,d\varepsilon
\right\}.
\]
If $d_*>4$, setting $\alpha=n^{-2/d_*}$ gives a bound of order $n^{-2/d_*}$. If $d_*=4$, choosing
$\alpha=n^{-1/2}$ gives a bound of order $n^{-1/2}\log(en)$. If $d_*<4$, the integral converges at $0$, so
the bound is $O(n^{-1/2})$. Since $C_{\mathrm{ext}}$ depends only on $r,k,d_x,d_y$, the displayed estimate
follows.
\end{proof}

\begin{proposition}\label{prop:g-entropy}
There exists a constant $C=C(r,k,d_x,d_y)$ such that
\[
\log \Cov\bigl(\varepsilon,\cG_Y,\|\cdot\|_{\infty,B_y}\bigr)
\le
C\,\varepsilon^{-d_*/2}
\qquad
\text{for all }0<\varepsilon\le 1.
\]
\end{proposition}

\begin{proof}
By Proposition~\ref{prop:G-covering} in Appendix~\ref{app:covering},
\[
\Cov\bigl(\varepsilon,\cG_Y,\|\cdot\|_{\infty,B_y}\bigr)
\le
\Cov\!\left(\frac{\varepsilon}{2L_{\mathrm{par}}},\Theta,\|\cdot\|_2\right)
\Cov\!\left(\frac{\varepsilon}{2},\cF_X,\|\cdot\|_{\infty,B_x}\right),
\]
where $L_{\mathrm{par}}$ depends only on $r,k,d_x,d_y$. The second factor is bounded by
$\exp(C\varepsilon^{-d_*/2})$ by Lemma~\ref{lem:semicolon-entropy}. The first factor contributes only a
finite-dimensional volumetric term of order $\log(C/\varepsilon)$ by Lemma~\ref{lem:volumetric-covering} in
Appendix~\ref{app:covering}. Since $\log(C/\varepsilon)\lesssim \varepsilon^{-d_*/2}$ for $0<\varepsilon\le 1$,
the claim follows.
\end{proof}

\begin{proposition}\label{prop:g-empirical}
There exists a constant $C=C(r,k,d_x,d_y)$ such that
\[
\E\sup_{g\in \cG_Y}|\Pair{\nu-\widehat\nu_n}{g}|
\le
C\,
\rho_n(d_*).
\]
\end{proposition}

\begin{proof}
Every $g\in \cG_Y$ is of the form $g=\mathsf C_\theta f$ with $\theta\in \Theta$ and $f\in \cF_X$. Since
\[
g(y)=\inf_{x\in B_x}\bigl(c_\theta(x,y)-f(x)\bigr),
\]
we have
\[
|g(y)|
\le
\sup_{(x,y)\in B_x\times B_y}|c_\theta(x,y)|+\|f\|_{L^\infty(B_x)}
\le
C_{\mathrm{cost}}+C_{\mathrm{ext}}
=:C_{\mathcal G}.
\]
Thus $\cG_Y$ is uniformly bounded, with $C_{\mathcal G}$ depending only on $r,k,d_x,d_y$.

\medskip

\noindent Applying Dudley's entropy integral inequality together with Proposition~\ref{prop:g-entropy} gives
\[
\E\sup_{g\in \cG_Y}|\Pair{\nu-\widehat\nu_n}{g}|
\le
\inf_{0<\alpha\le 1}
\left\{
\alpha+\frac{C}{\sqrt n}\int_\alpha^1 \varepsilon^{-d_*/4}\,d\varepsilon
\right\}
+
\frac{CC_{\mathcal G}}{\sqrt n}.
\]
The same case analysis as in Proposition~\ref{prop:f-empirical} gives the stated rate.
\end{proof}

\begin{proposition}[Transport empirical-process bound]\label{prop:transport-final}
There exists a constant $C=C(r,k,d_x,d_y)$ such that
\[
\E\Bigl|
\CGW_{r,k}(\mu,\nu)-\CGW_{r,k}(\widetilde\mu_n,\widetilde\nu_n)
\Bigr|
\le
C\,
\rho_n(d_*).
\]
\end{proposition}

\begin{proof}
Combine Proposition~\ref{prop:transport-reduction} with Propositions~\ref{prop:f-empirical} and
\ref{prop:g-empirical}. The additional centering term from Proposition~\ref{prop:transport-reduction} is of
order $n^{-1/2}$, which is absorbed by the displayed rate for every value of $d_*$, because
\[
\frac1{\sqrt n}
\le
\rho_n(d_*)
\qquad
(n\ge 1).
\]
\end{proof}

\begin{proof}[Proof of Theorem~\ref{thm:main}]
Under the normalized assumptions \eqref{eq:normalized-hypotheses}, taking expectations in
Proposition~\ref{prop:centering-decomp} and using \eqref{eq:marginal-final} together with
Proposition~\ref{prop:transport-final} gives
\[
\E\Bigl|
\DGW_{r,k}(\mu,\nu)
-
\DGW_{r,k}(\widehat\mu_n,\widehat\nu_n)
\Bigr|
\le
C_{r,k,d_x,d_y}
\left[
\frac1{\sqrt n}
+
\rho_n(d)
\right].
\]
Applying this bound to the normalized pair $(\mu^\circ,\nu^\circ)$ and using \eqref{eq:rescaling-main}, we
obtain
\[
\E\Bigl|
\DGW_{r,k}(\mu,\nu)
-
\DGW_{r,k}(\widehat\mu_n,\widehat\nu_n)
\Bigr|
\le
C_{r,k,d_x,d_y}
(2R)^{4kr}
\left[
\frac1{\sqrt n}
+
\rho_n(d_\ast)
\right].
\]
Absorbing $2^{4kr}$ into the constant proves the theorem.
\end{proof}

\section{Lower bound}

We prove a lower bound for the case when $d_x \wedge d_y < 4$, and thus establish sharpness when $d_x \wedge d_y < 4$. The other cases are left for future work. 

\medskip

\begin{proposition}[Parametric lower bound]
\label{prop:parametric-lower}
Fix \(r,k\ge 1\), \(d_x,d_y\ge 1\), and \(R>0\). There exists a constant \(c>0\) such that, for all sufficiently large \(n\),
\[
\sup_{\substack{
\operatorname{diam}(\operatorname{supp}\mu)\le R\\
\operatorname{diam}(\operatorname{supp}\nu)\le R
}}
\mathbb E\left|
\DGW_{r,k}(\widehat\mu_n,\widehat\nu_n)
-
\DGW_{r,k}(\mu,\nu)
\right|
\ge
cR^{4kr}n^{-1/2}.
\]
Consequently, when \(d_x\wedge d_y<4\), the rate in Theorem~\ref{thm:main} is sharp.
\end{proposition}

\begin{proof}
Fix $R > 0$, set $p = 1/4$, and let
\(e_1\in\mathbb R^{d_x}\) be the first coordinate unit vector, and set
\[
\mu=(1-p)\delta_0+p\delta_{Re_1},
\qquad
\nu=\delta_0.
\]
Then
\[
\operatorname{diam}(\operatorname{supp}\mu)=R,
\qquad
\operatorname{diam}(\operatorname{supp}\nu)=0.
\]
Since \(\nu=\delta_0\), the only coupling between \(\mu\) and \(\nu\) is
\(\mu\otimes\delta_0\). Hence
\[
\DGW_{r,k}(\mu,\nu)
=
\iint \|x-x'\|^{4kr}\,d\mu(x)d\mu(x').
\]
If \(X,X'\sim\mu\) independently, then $\mathbb{P}(\|X-X'\|=R) = 2p(1-p)$. Therefore
\[
\DGW_{r,k}(\mu,\nu)
=
2p(1-p)R^{4kr}.
\]
Set
\[
g(t):=2t(1-t).
\]
Then
\[
\DGW_{r,k}(\mu,\nu)=g(p)R^{4kr}.
\]
Let
\[
\widehat p_n:=\widehat\mu_n(\{Re_1\}).
\]
Then \(n\widehat p_n\sim\operatorname{Bin}(n,p)\), and since
\(\widehat\nu_n=\delta_0\) almost surely,
\[
\DGW_{r,k}(\widehat\mu_n,\widehat\nu_n)
=
2\widehat p_n(1-\widehat p_n)R^{4kr}
=
g(\widehat p_n)R^{4kr}.
\]
Thus
\[
\mathbb E\left|
\DGW_{r,k}(\widehat\mu_n,\widehat\nu_n)
-
\DGW_{r,k}(\mu,\nu)
\right|
=
R^{4kr}\mathbb E|g(\widehat p_n)-g(p)|.
\]
Set \(\Delta_n:=\widehat p_n-p\). Taylor's expansion of $g(t)$ around $p$ gives:
\[
g(\widehat p_n)-g(p)
=
g'(p)\Delta_n-2\Delta_n^2,
\qquad
g'(p)=2(1-2p)\neq0.
\]
Hence
\[
\mathbb E|g(\widehat p_n)-g(p)|
\ge
|g'(p)|\,\mathbb E|\Delta_n|-2\mathbb E\Delta_n^2.
\]
We know that
\[
\mathbb E\Delta_n^2=\operatorname{Var}(\widehat p_n)=\frac{p(1-p)}{n}.
\]
Furthermore, by the central limit theorem and uniform integrability,
\[
\sqrt n\,\mathbb E|\widehat p_n-p|
\longrightarrow
\sqrt{\frac{2p(1-p)}{\pi}}.
\]
Therefore, for all sufficiently large \(n\),
\[
\mathbb E|\Delta_n|\ge c_p n^{-1/2}
\]
for some \(c_p>0\). Thus,
\[
\mathbb E\left|
\DGW_{r,k}(\widehat\mu_n,\widehat\nu_n)
-
\DGW_{r,k}(\mu,\nu)
\right|
= R^{4kr}\mathbb E|g(\widehat p_n)-g(p)|
\ge c'_p R^{4kr} n^{-1/2}.
\]
Thus, we get the lower bound. If \(d:=d_x\wedge d_y<4\), then the upper bound from Theorem~\ref{thm:main} gives
\[
\sup_{\mu,\nu}
\mathbb E\left|
\DGW_{r,k}(\widehat\mu_n,\widehat\nu_n)
-
\DGW_{r,k}(\mu,\nu)
\right|
\lesssim
R^{4kr}n^{-2/(d\vee4)}
=
R^{4kr}n^{-1/2}.
\]
This establishes sharpness for \(d_x\wedge d_y<4\).
\end{proof}

\section*{Acknowledgments}

The author is grateful to Professor Oanh Nguyen for suggesting the problems studied in this paper
and for many helpful discussions.

\appendix

\section{\texorpdfstring{Generalized duality for the case $(2r,2k)$}{Generalized duality for the case (2r,2k)}}\label{app:duality}

This appendix proves Theorem~\ref{thm:duality-main}. The proof below follows the generalized duality argument of Zhang, Goldfeld,
Mroueh, and Sriperumbudur \cite[Appendix G]{ZhangEtAl}.
Their appendix treats the \((2,2k)\) case and notes that the same polynomial
linearization method extends to even \(p\). We record the \((2r,2k)\) version needed
here, with the relevant bookkeeping and compactness details made explicit.

\begin{proposition}[Polynomial decomposition]\label{prop:poly-decomp}
There exist an integer $J\ge 0$, an integer $\ell\in \{0,\dots,J\}$, and polynomials
\[
P_1,\dots,P_J:\R^{d_x}\times \R^{d_y}\to \R,
\]
each of degree at most $4kr$, depending only on $r,k,d_x,d_y$, such that
\[
\DGW_{r,k}(\mu,\nu)
=
\MGW_{r,k}(\mu,\nu)
+
\inf_{\pi\in \Pi(\mu,\nu)}
\left[
\sum_{i=1}^{\ell}\left(\int P_i\,d\pi\right)^2
-
\sum_{i=\ell+1}^{J}\left(\int P_i\,d\pi\right)^2
\right].
\]
Moreover, $\MGW_{r,k}(\mu,\nu)$ depends only on moments of $\mu$ and $\nu$ of order at most $4kr$.
\end{proposition}

\begin{proof}
Set
\[
K_X(x,x'):=\|x-x'\|^{2k},
\qquad
K_Y(y,y'):=\|y-y'\|^{2k}.
\]
By the binomial theorem,
\[
\bigl(K_X(x,x')-K_Y(y,y')\bigr)^{2r}
=
\sum_{s=0}^{2r}(-1)^s\binom{2r}{s}\bigl(K_X(x,x')\bigr)^s\bigl(K_Y(y,y')\bigr)^{2r-s}.
\]
Expanding each factor by the multinomial theorem gives
\[
\bigl(K_X(x,x')\bigr)^s
=
\sum_{|\alpha|+|\beta|=2ks}
p_{\alpha\beta}^{(s)}x^\alpha(x')^\beta,
\qquad
\bigl(K_Y(y,y')\bigr)^{2r-s}
=
\sum_{|\gamma|+|\delta|=2k(2r-s)}
q_{\gamma\delta}^{(s)}y^\gamma(y')^\delta,
\]
with coefficients depending only on $r,k,d_x,d_y$. Therefore
\[
\DGW_{r,k}(\mu,\nu)
=
\inf_{\pi\in \Pi(\mu,\nu)}
\sum_{s=0}^{2r}
(-1)^s\binom{2r}{s}
\!\!\!\!\!\!
\sum_{\substack{|\alpha|+|\beta|=2ks\\ |\gamma|+|\delta|=2k(2r-s)}}
\!\!\!\!\!\!
p_{\alpha\beta}^{(s)}q_{\gamma\delta}^{(s)}
\left(\int x^\alpha y^\gamma\,d\pi\right)
\left(\int x^\beta y^\delta\,d\pi\right).
\]
Let $\MGW_{r,k}(\mu,\nu)$ collect all marginal-only terms, and let $\QGW_{r,k}(\pi)$ denote the rest.
Then
\[
\DGW_{r,k}(\mu,\nu)
=
\MGW_{r,k}(\mu,\nu)
+
\inf_{\pi\in \Pi(\mu,\nu)}\QGW_{r,k}(\pi).
\]
Let $\mathfrak J$ be the finite set consisting of $(0,0)$ together with every pair $(\alpha,\gamma)$ that appears in
some factor $\int x^\alpha y^\gamma\,d\pi$ occurring in $\QGW_{r,k}(\pi)$. For $j=(\alpha,\gamma)\in \mathfrak J$,
define
\[
h_j(x,y):=x^\alpha y^\gamma,
\qquad
u_j(\pi):=\int h_j(x,y)\,d\pi(x,y).
\]
Because $(0,0)\in \mathfrak J$, the constant function $1$ is included among the $h_j$, so affine terms in the mixed
moments are absorbed into the same quadratic form. Hence there exists a real matrix $C_0$ such that
\[
\QGW_{r,k}(\pi)=u(\pi)^\top C_0u(\pi),
\qquad
u(\pi):=(u_j(\pi))_{j\in \mathfrak J}.
\]
Replacing $C_0$ with its symmetrization $C:=(C_0+C_0^\top)/2$, we may assume that $C$ is symmetric and
still have
\[
\QGW_{r,k}(\pi)=u(\pi)^\top C u(\pi).
\]
Let $m_0:=|\mathfrak J|$. Since $C$ is symmetric, there exist an orthogonal matrix $U$ and real eigenvalues
$\lambda_1,\dots,\lambda_{m_0}$ such that
\[
C=U^\top \mathrm{diag}(\lambda_1,\dots,\lambda_{m_0})U.
\]
Discarding the zero eigenvalues and reordering if necessary, we may assume
\[
\lambda_1,\dots,\lambda_\ell>0,
\qquad
\lambda_{\ell+1},\dots,\lambda_J<0
\]
for some $J\le m_0$. Define
\[
P_i(x,y):=
\sqrt{\lambda_i}
\sum_{j\in \mathfrak J}U_{ij}h_j(x,y),
\qquad
1\le i\le \ell,
\]
and
\[
P_i(x,y):=
\sqrt{-\lambda_i}
\sum_{j\in \mathfrak J}U_{ij}h_j(x,y),
\qquad
\ell+1\le i\le J.
\]
Each $P_i$ is a polynomial of degree at most $4kr$. Moreover, for every $\pi\in \Pi(\mu,\nu)$,
\[
\QGW_{r,k}(\pi)
=
\sum_{i=1}^{\ell}\left(\int P_i\,d\pi\right)^2
-
\sum_{i=\ell+1}^{J}\left(\int P_i\,d\pi\right)^2.
\]
Substituting this identity into the preceding decomposition proves the claim.
\end{proof}

\begin{proposition}[Compact dual representation]\label{prop:compact-dual}
Let $P_1,\dots,P_J$ and $\ell$ be as in Proposition~\ref{prop:poly-decomp}. For
\[
u\in \R^\ell,
\qquad
v\in \R^{J-\ell},
\]
define
\[
c_{u,v}(x,y)
=
\sum_{i=1}^{\ell}u_i P_i(x,y)
-
\sum_{j=1}^{J-\ell}v_j P_{\ell+j}(x,y).
\]
For each $i=1,\dots,J$, let
\[
\kappa_i
:=
\sup_{\pi\in \Pi(\mu,\nu)}
\left|
\int P_i\,d\pi
\right|,
\]
and define
\[
\Theta_+
:=
\prod_{i=1}^{\ell}\left[-\frac{\kappa_i}{2},\frac{\kappa_i}{2}\right],
\qquad
\Theta_-
:=
\prod_{j=1}^{J-\ell}\left[-\frac{\kappa_{\ell+j}}{2},\frac{\kappa_{\ell+j}}{2}\right].
\]
Then $\kappa_i<\infty$ for every $i$, and
\[
\DGW_{r,k}(\mu,\nu)
=
\MGW_{r,k}(\mu,\nu)
+
4\inf_{v\in \Theta_-}\sup_{u\in \Theta_+}
\Bigl\{
-\|u\|^2+\|v\|^2+\OT_{c_{u,v}}(\mu,\nu)
\Bigr\}.
\]
\end{proposition}

\begin{proof}
For $\pi\in \Pi(\mu,\nu)$, set
\[
m_i(\pi):=\int P_i\,d\pi,
\qquad
i=1,\dots,J.
\]
By Proposition~\ref{prop:poly-decomp},
\[
\DGW_{r,k}(\mu,\nu)
=
\MGW_{r,k}(\mu,\nu)
+
\inf_{\pi\in \Pi(\mu,\nu)}
\left[
\sum_{i=1}^{\ell}m_i(\pi)^2
-
\sum_{j=1}^{J-\ell}m_{\ell+j}(\pi)^2
\right].
\]
Since each $P_i$ is a polynomial of degree at most $4kr$, there exists $C_i>0$ such that
\[
|P_i(x,y)|
\le
C_i\bigl(1+\|x\|^{4kr}+\|y\|^{4kr}\bigr)
\qquad
\text{for all }(x,y)\in \R^{d_x}\times \R^{d_y}.
\]
Hence, for every $\pi\in \Pi(\mu,\nu)$,
\[
\left|
\int P_i\,d\pi
\right|
\le
C_i
\left(
1+\int \|x\|^{4kr}\,d\mu(x)+\int \|y\|^{4kr}\,d\nu(y)
\right),
\]
so $\kappa_i<\infty$.

\medskip

\noindent Define
\[
Q(m)
:=
\sum_{i=1}^{\ell}m_i^2-\sum_{j=1}^{J-\ell}m_{\ell+j}^2,
\qquad
m=(m_1,\dots,m_J)\in \R^J.
\]
Let
\[
\mathscr M
:=
\Bigl\{
(m_1(\pi),\dots,m_J(\pi)):\pi\in \Pi(\mu,\nu)
\Bigr\}
\subset \R^J.
\]
Then
\[
\DGW_{r,k}(\mu,\nu)
=
\MGW_{r,k}(\mu,\nu)+\inf_{m\in \mathscr M}Q(m).
\]
We claim that $\mathscr M$ is compact and convex. Convexity is follows from the convexity of $\Pi(\mu,\nu)$ and the
linearity of the moment map. For compactness, note first that $\Pi(\mu,\nu)$ is weakly compact: it is tight
because $\mu$ and $\nu$ are tight, and it is weakly closed because weak limits preserve the marginals. It
therefore suffices to show that each map
\[
\pi\longmapsto \int P_i\,d\pi
\]
is continuous on $\Pi(\mu,\nu)$.

\medskip

\noindent Fix $i$. Choose a continuous cutoff $\eta_R:\R^{d_x}\times \R^{d_y}\to [0,1]$ such that $\eta_R=1$ on
$B_{d_x}(0,R)\times B_{d_y}(0,R)$ and $\eta_R=0$ outside
$B_{d_x}(0,2R)\times B_{d_y}(0,2R)$, and set
\[
P_i^{(R)}:=P_i\eta_R.
\]
If $\pi_n\Rightarrow \pi$, then $P_i^{(R)}$ is bounded and continuous, so
\[
\int P_i^{(R)}\,d\pi_n \longrightarrow \int P_i^{(R)}\,d\pi.
\]

On the other hand, using the polynomial growth bound and the fact that every
\(\sigma\in\Pi(\mu,\nu)\) has marginals \(\mu,\nu\), we get the following uniform
tail estimate. Set \(p:=4kr\),
\[
A_R:=\{x:\|x\|>R\}, \qquad B_R:=\{y:\|y\|>R\}.
\]
Since \(P_i\) has degree at most \(p\), there exists \(C_i>0\) such that
\[
|P_i(x,y)|\le C_i\bigl(1+\|x\|^p+\|y\|^p\bigr).
\]
Moreover, for \(R\ge 1\),
\[
\bigl(1+\|x\|^p+\|y\|^p\bigr)\mathbf 1_{A_R\cup B_R}
\le
2(1+\|x\|^p)\mathbf 1_{A_R}
+
2(1+\|y\|^p)\mathbf 1_{B_R}.
\]
Therefore
\[
\begin{aligned}
\sup_{\sigma\in\Pi(\mu,\nu)}
\int |P_i-P_i^{(R)}|\,d\sigma
&\le
C_i
\sup_{\sigma\in\Pi(\mu,\nu)}
\int_{A_R\cup B_R}
\bigl(1+\|x\|^p+\|y\|^p\bigr)\,d\sigma(x,y)
\\
&\le
2C_i
\left[
\int_{A_R}(1+\|x\|^p)\,d\mu(x)
+
\int_{B_R}(1+\|y\|^p)\,d\nu(y)
\right].
\end{aligned}
\]
The right-hand side tends to \(0\) as \(R\to\infty\), since
\(\mu,\nu\in\mathcal P_p\). Now suppose \(\pi_n\Rightarrow \pi\) weakly in \(\Pi(\mu,\nu)\). For every fixed
\(R\), the function \(P_i^{(R)}\) is bounded and continuous, hence
\[
\int P_i^{(R)}\,d\pi_n \to \int P_i^{(R)}\,d\pi.
\]
Therefore
\[
\begin{aligned}
\limsup_{n\to\infty}
\left|
\int P_i\,d\pi_n-\int P_i\,d\pi
\right|
&\le
\limsup_{n\to\infty}
\left|
\int (P_i-P_i^{(R)})\,d\pi_n
\right|
\\
&\qquad
+
\left|
\int (P_i-P_i^{(R)})\,d\pi
\right|
\\
&\le
2\sup_{\sigma\in\Pi(\mu,\nu)}
\int |P_i-P_i^{(R)}|\,d\sigma.
\end{aligned}
\]
Letting \(R\to\infty\) proves
\[
\int P_i\,d\pi_n \to \int P_i\,d\pi.
\]
Thus each coordinate of the moment map is continuous on \(\Pi(\mu,\nu)\), so
\(\mathscr M\) is compact.

\medskip

\noindent Now define, for $m\in \mathscr M$, $u\in \Theta_+$, and $v\in \Theta_-$,
\[
\Psi(m,u,v)
:=
-\|u\|^2+\|v\|^2
+
\sum_{i=1}^{\ell}u_i m_i
-
\sum_{j=1}^{J-\ell}v_j m_{\ell+j}.
\]
For each fixed $m\in \mathscr M$, the optimization is coordinate-wise:
\[
\Psi(m,u,v)
=
\sum_{i=1}^{\ell}\bigl(-u_i^2+u_i m_i\bigr)
+
\sum_{j=1}^{J-\ell}\bigl(v_j^2-v_j m_{\ell+j}\bigr).
\]
Since $|m_i|\le \kappa_i$ on $\mathscr M$, the maximizer of $t\mapsto -t^2+t\,m_i$ is $t=m_i/2\in [-\kappa_i/2,\kappa_i/2]$, and
the minimizer of $t\mapsto t^2-t\,m_{\ell+j}$ is $t=m_{\ell+j}/2\in [-\kappa_{\ell+j}/2,\kappa_{\ell+j}/2]$. Hence
\[
\sup_{u\in \Theta_+}\inf_{v\in \Theta_-}\Psi(m,u,v)
=
\inf_{v\in \Theta_-}\sup_{u\in \Theta_+}\Psi(m,u,v)
=
\frac14\,Q(m).
\]
Therefore
\[
\DGW_{r,k}(\mu,\nu)
=
\MGW_{r,k}(\mu,\nu)
+
4\inf_{m\in \mathscr M}\inf_{v\in \Theta_-}\sup_{u\in \Theta_+}\Psi(m,u,v).
\]
Since the two infima commute,
\[
\DGW_{r,k}(\mu,\nu)
=
\MGW_{r,k}(\mu,\nu)
+
4\inf_{v\in \Theta_-}\inf_{m\in \mathscr M}\sup_{u\in \Theta_+}\Psi(m,u,v).
\]
Fix $v\in \Theta_-$. The function
\[
(m,u)\longmapsto \Psi(m,u,v)
\]
is continuous, affine in $m$, and concave in $u$. Because $\mathscr M$ and $\Theta_+$ are compact and convex, Sion's
minimax theorem \cite{Sion} gives
\[
\inf_{m\in \mathscr M}\sup_{u\in \Theta_+}\Psi(m,u,v)
=
\sup_{u\in \Theta_+}\inf_{m\in \mathscr M}\Psi(m,u,v).
\]
Hence
\[
\DGW_{r,k}(\mu,\nu)
=
\MGW_{r,k}(\mu,\nu)
+
4\inf_{v\in \Theta_-}\sup_{u\in \Theta_+}\inf_{m\in \mathscr M}\Psi(m,u,v).
\]
Finally, because $\mathscr M$ is the image of $\Pi(\mu,\nu)$ under the moment map,
\[
\inf_{m\in \mathscr M}\Psi(m,u,v)
=
-\|u\|^2+\|v\|^2+\inf_{\pi\in \Pi(\mu,\nu)}\int c_{u,v}\,d\pi
=
-\|u\|^2+\|v\|^2+\OT_{c_{u,v}}(\mu,\nu).
\]
Therefore, 
\[
\DGW_{r,k}(\mu,\nu)
=
\MGW_{r,k}(\mu,\nu)
+
4\inf_{v\in \Theta_-}\sup_{u\in \Theta_+}
\Bigl\{
-\|u\|^2+\|v\|^2+\OT_{c_{u,v}}(\mu,\nu)
\Bigr\}.
\]
\end{proof}

\bigskip

\begin{proof}[Proof of Theorem~\ref{thm:duality-main}]
Proposition~\ref{prop:poly-decomp} gives the polynomial decomposition of the coupling term, and
Proposition~\ref{prop:compact-dual} turns that decomposition into the compact dual formula. The displayed
formula for $\CGW_{r,k}(\mu,\nu)$ is the second display in Proposition~\ref{prop:compact-dual}.
\end{proof}

\begin{proof}[Proof of Corollary~\ref{cor:uniform-boxes}]
Let $P_1,\dots,P_J$ and $\ell$ be as in Proposition~\ref{prop:poly-decomp}. For each $i=1,\dots,J$, define
\[
C_i^{\mathrm{basis}}
:=
\sup\Bigl\{
|P_i(x,y)|:
x\in B_{d_x}(0,\rho_x),\ y\in B_{d_y}(0,\rho_y)
\Bigr\}.
\]
These numbers are finite because each $P_i$ is continuous on the compact set
\[
B_{d_x}(0,\rho_x)\times B_{d_y}(0,\rho_y),
\]
and they depend only on $r,k,d_x,d_y,\rho_x,\rho_y$.

\medskip

\noindent Now set
\[
\widetilde\Theta_+
:=
\prod_{i=1}^{\ell}
\left[-\frac{C_i^{\mathrm{basis}}}{2},\frac{C_i^{\mathrm{basis}}}{2}\right],
\qquad
\widetilde\Theta_-
:=
\prod_{j=1}^{J-\ell}
\left[-\frac{C_{\ell+j}^{\mathrm{basis}}}{2},\frac{C_{\ell+j}^{\mathrm{basis}}}{2}\right].
\]
These are compact sets depending only on $r,k,d_x,d_y,\rho_x,\rho_y$. Let $\mu,\nu$ satisfy
\[
\supp(\mu)\subset B_{d_x}(0,\rho_x),
\qquad
\supp(\nu)\subset B_{d_y}(0,\rho_y).
\]
If $\pi\in \Pi(\mu,\nu)$, then $\pi$ is concentrated on
\[
B_{d_x}(0,\rho_x)\times B_{d_y}(0,\rho_y),
\]
so for every $i=1,\dots,J$,
\[
\left|\int P_i\,d\pi\right|
\le
\int |P_i|\,d\pi
\le
C_i^{\mathrm{basis}}.
\]
Set
\[
m_i(\pi):=\int P_i\,d\pi,
\]
We therefore have
\[
|m_i(\pi)|\le C_i^{\mathrm{basis}}
\qquad
\text{for all }i=1,\dots,J.
\]
For fixed $\pi$, define
\[
\Phi_\pi(u,v)
:=
-\|u\|^2+\|v\|^2+\int c_{u,v}\,d\pi.
\]
As in the proof of Proposition~\ref{prop:compact-dual},
\[
\Phi_\pi(u,v)
=
\sum_{i=1}^{\ell}\bigl(-u_i^2+u_i m_i(\pi)\bigr)
+
\sum_{j=1}^{J-\ell}\bigl(v_j^2-v_j m_{\ell+j}(\pi)\bigr).
\]
Hence the coordinate-wise maximizer in the variable $u_i$ is $m_i(\pi)/2$, and the coordinate-wise minimizer
in the variable $v_j$ is $m_{\ell+j}(\pi)/2$. Because
\[
|m_i(\pi)|\le C_i^{\mathrm{basis}},
\]
all of these optimizers belong to $\widetilde\Theta_+$ and $\widetilde\Theta_-$. Therefore, for every
$\pi\in \Pi(\mu,\nu)$,
\[
\sup_{u\in \widetilde\Theta_+}\inf_{v\in \widetilde\Theta_-}\Phi_\pi(u,v)
=
\inf_{v\in \widetilde\Theta_-}\sup_{u\in \widetilde\Theta_+}\Phi_\pi(u,v)
=
\frac14
\left(
\sum_{i=1}^{\ell}m_i(\pi)^2-\sum_{j=1}^{J-\ell}m_{\ell+j}(\pi)^2
\right).
\]
Repeating the remainder of the proof of Proposition~\ref{prop:compact-dual} with $\widetilde\Theta_+$ and
$\widetilde\Theta_-$ in place of the measure-dependent boxes,
\[
\DGW_{r,k}(\mu,\nu)
=
\MGW_{r,k}(\mu,\nu)
+
4\inf_{v\in \widetilde\Theta_-}\sup_{u\in \widetilde\Theta_+}
\Bigl\{
-\|u\|^2+\|v\|^2+\OT_{c_{u,v}}(\mu,\nu)
\Bigr\}.
\]
Thus the compact sets in Theorem~\ref{thm:duality-main} may be chosen to depend only on
$r,k,d_x,d_y,\rho_x,\rho_y$.
\end{proof}

\section{\texorpdfstring{Polynomial bookkeeping for $\MGW_{r,k}$}{Polynomial bookkeeping for S1}}\label{app:marginal}

This appendix collects some algebraic facts about the marginal term $\MGW_{r,k}$.

\begin{proof}[Proof of Proposition~\ref{prop:marginal-structure}]
Fix $s\in \{0,\dots,2r\}$. In the multinomial expansion of
\[
\|x-x'\|^{2ks}\|y-y'\|^{2k(2r-s)},
\]
every term has the form
\[
p_{\alpha\beta}^{(s)}q_{\gamma\delta}^{(s)}
\left(\int x^\alpha y^\gamma\,d\pi\right)
\left(\int x^\beta y^\delta\,d\pi\right),
\]
with
\[
|\alpha|+|\beta|=2ks,
\qquad
|\gamma|+|\delta|=2k(2r-s).
\]
If the term is marginal-only, then each factor depends only on one marginal. Therefore each factor is one of
\[
M^X_\alpha(\mu),
\qquad
M^X_\beta(\mu),
\qquad
M^Y_\gamma(\nu),
\qquad
M^Y_\delta(\nu),
\]
and the corresponding monomial is one of
\[
M^X_\alpha(\mu)M^X_\beta(\mu),
\qquad
M^Y_\gamma(\nu)M^Y_\delta(\nu),
\qquad
M^X_\alpha(\mu)M^Y_\delta(\nu),
\qquad
M^Y_\gamma(\nu)M^X_\beta(\mu),
\]
or a degenerate case with one factor equal to $1$. Each individual moment has order at most $4kr$, and each
monomial contains at most two such moments. Summing over the finite index set of all marginal-only terms
proves the claim.
\end{proof}

\begin{proof}[Proof of Proposition~\ref{prop:marginal-emp-vs-centered}]
Assume the normalized hypotheses~(\ref{eq:normalized-hypotheses}). By Proposition~\ref{prop:marginal-structure}, it suffices to control the centering error for a single marginal moment and for a product of two marginal
moments.

\smallskip

For a multi-index $\alpha\in\mathbb N_0^{d_x}$, by multinomial theorem,
\[
M^X_\alpha(\widetilde\mu_n)
= \frac1n\sum_{i=1}^n (X_i-\overline X_n)^\alpha
= \sum_{\beta\le \alpha} \binom{\alpha}{\beta}(-\overline X_n)^{\alpha-\beta}M^X_\beta(\widehat\mu_n),
\]
where $\beta\le \alpha$ is understood componentwise. Hence
\[
M^X_\alpha(\widetilde\mu_n)-M^X_\alpha(\widehat\mu_n)
= \sum_{\beta<\alpha} \binom{\alpha}{\beta}(-\overline X_n)^{\alpha-\beta}M^X_\beta(\widehat\mu_n).
\]
Since $\widehat\mu_n$ is supported in $B_{d_x}(0,1)$,
\[
|M^X_\beta(\widehat\mu_n)|\le 1.
\]
Also, because $\|\overline X_n\|\le 1$ and $|\alpha-\beta|\ge 1$ for every $\beta<\alpha$,
\[
|(-\overline X_n)^{\alpha-\beta}|
\le \|\overline X_n\|^{|\alpha-\beta|}
\le \|\overline X_n\|.
\]
Therefore
\[
|M^X_\alpha(\widetilde\mu_n)-M^X_\alpha(\widehat\mu_n)|
\le \Big(\sum_{\beta<\alpha}\binom{\alpha}{\beta}\Big)\|\overline X_n\|
= (2^{|\alpha|}-1)\|\overline X_n\|.
\]
The same argument on the $\nu$-side gives, for every $\gamma\in\mathbb N_0^{d_y}$,
\[
|M^Y_\gamma(\widetilde\nu_n)-M^Y_\gamma(\widehat\nu_n)|
\le (2^{|\gamma|}-1)\|\overline Y_n\|.
\]
Moreover, since $\widetilde\mu_n$ and $\widetilde\nu_n$ are supported in radius-$2$ balls,
\[
|M^X_\alpha(\widetilde\mu_n)|\le 2^{|\alpha|},
\qquad
|M^Y_\gamma(\widetilde\nu_n)|\le 2^{|\gamma|},
\]
while $|M^X_\alpha(\widehat\mu_n)|\le 1$ and $|M^Y_\gamma(\widehat\nu_n)|\le 1$. Now let $U$ be a monomial appearing in $\MGW_{r,k}$. If $U$ is linear in a single moment, the preceding bounds give
\[
|U(\widehat\mu_n,\widehat\nu_n)-U(\widetilde\mu_n,\widetilde\nu_n)|
\le C_{r,k}(\|\overline X_n\|+\|\overline Y_n\|).
\]
If $U$ is a product of two moments, set $U=ab$, where $a,a'$ are the corresponding first moments evaluated at $(\widehat\mu_n,\widehat\nu_n)$ and $(\widetilde\mu_n,\widetilde\nu_n)$, and likewise $b,b'$ for the second factor. Then
\[
|ab-a'b'|\le |a-a'|\,|b| + |a'|\,|b-b'|.
\]
The factor $|b|$ is bounded by $1$, and the factor $|a'|$ is bounded by $2^{4kr}$ because each individual moment order is at most $4kr$. Combining this with the one-moment centering bounds above shows that
\[
|U(\widehat\mu_n,\widehat\nu_n)-U(\widetilde\mu_n,\widetilde\nu_n)|
\le C_{r,k}(\|\overline X_n\|+\|\overline Y_n\|).
\]
Summing over the finitely many monomials in $\MGW_{r,k}$ gives
\[
|\MGW_{r,k}(\widehat\mu_n,\widehat\nu_n)-\MGW_{r,k}(\widetilde\mu_n,\widetilde\nu_n)|
\le C_{r,k,d_x,d_y}(\|\overline X_n\|+\|\overline Y_n\|).
\]
Finally,
\[
\mathbb E\|\overline X_n\|
\le \bigl(\mathbb E\|\overline X_n\|^2\bigr)^{1/2}
= \Bigl(\frac1n\mathbb E\|X_1\|^2\Bigr)^{1/2}
\le \frac1{\sqrt n},
\qquad
\mathbb E\|\overline Y_n\|
\le \frac1{\sqrt n},
\]
because $\mathbb E X_1=0$, $\mathbb E Y_1=0$, and $\|X_1\|,\|Y_1\|\le 1$ almost surely under the normalized hypotheses~(\ref{eq:normalized-hypotheses}). Hence,
\[
\mathbb E\bigl|\MGW_{r,k}(\widehat\mu_n,\widehat\nu_n)-\MGW_{r,k}(\widetilde\mu_n,\widetilde\nu_n)\bigr|
\le \frac{C_{r,k,d_x,d_y}}{\sqrt n}.
\]
\end{proof}

\section{Covering lemmas for the second-potential class}\label{app:covering}

This appendix contains the routine finite-dimensional covering facts used to control the class $\cG_Y$.

\begin{lemma}\label{lem:Ttheta-contraction}
Fix $\theta\in \Theta$. Then
\[
\mathsf C_\theta:\cF_X\to C(B_y)
\]
is $1$-Lipschitz with respect to the sup norm. That is, for all $f_1,f_2\in \cF_X$,
\[
\|\mathsf C_\theta f_1-\mathsf C_\theta f_2\|_{\infty,B_y}
\le
\|f_1-f_2\|_{\infty,B_x}.
\]
\end{lemma}

\begin{proof}
Let $\delta:=\|f_1-f_2\|_{\infty,B_x}$. Then $f_1(x)\le f_2(x)+\delta$ for all $x\in B_x$, hence
\[
c_\theta(x,y)-f_1(x)
\ge
c_\theta(x,y)-f_2(x)-\delta
\]
for every $(x,y)\in B_x\times B_y$. Taking the infimum over $x\in B_x$ gives
\[
\mathsf C_\theta f_1(y)\ge \mathsf C_\theta f_2(y)-\delta
\qquad
\text{for all }y\in B_y.
\]
Interchanging $f_1$ and $f_2$ gives the claim.
\end{proof}

\begin{lemma}\label{lem:volumetric-covering}
Assume that
\[
\Theta\subset B_m(0,R_\Theta)
\]
for some $m\in \N$ and $R_\Theta>0$. Then for every $\delta>0$,
\[
\Cov(\delta,\Theta,\|\cdot\|_2)
\le
\left(1+\frac{2R_\Theta}{\delta}\right)^m.
\]
\end{lemma}

\begin{proof}
Choose a maximal $\delta$-separated subset $\{\theta_1,\dots,\theta_N\}\subset \Theta$. The balls
$B_m(\theta_i,\delta/2)$ are pairwise disjoint and all contained in $B_m(0,R_\Theta+\delta/2)$. Comparing
volumes gives
\[
N\,\mathrm{Vol}(B_m(0,\delta/2))
\le
\mathrm{Vol}(B_m(0,R_\Theta+\delta/2)),
\]
hence
\[
N
\le
\left(1+\frac{2R_\Theta}{\delta}\right)^m.
\]
Maximality also implies that $\{\theta_i\}$ is a $\delta$-net of $\Theta$, so the same bound holds for the
covering number.
\end{proof}

\begin{lemma}\label{lem:parametric-lipschitz}
Assume that
\[
c_{u,v}(x,y)
=
\sum_{i=1}^{\ell}u_i P_i(x,y)
-
\sum_{j=1}^{J-\ell}v_j P_{\ell+j}(x,y),
\qquad
\theta=(u,v)\in \Theta.
\]
Set
\[
C_{\mathrm{basis}}:=\max_{1\le i\le J}\|P_i\|_{\infty,B_x\times B_y},
\qquad
L_{\mathrm{par}}:=\max\{1,C_{\mathrm{basis}}\sqrt J\}.
\]
Then for all $\theta,\theta'\in \Theta$,
\[
\|c_\theta-c_{\theta'}\|_{\infty,B_x\times B_y}
\le
L_{\mathrm{par}}\|\theta-\theta'\|_2.
\]
Consequently, for every $f\in C(B_x)$,
\[
\|\mathsf C_\theta f-\mathsf C_{\theta'}f\|_{\infty,B_y}
\le
L_{\mathrm{par}}\|\theta-\theta'\|_2.
\]
\end{lemma}

\begin{proof}
For $(x,y)\in B_x\times B_y$,
\[
c_\theta(x,y)-c_{\theta'}(x,y)
=
\sum_{i=1}^{\ell}(u_i-u_i')P_i(x,y)
-
\sum_{j=1}^{J-\ell}(v_j-v_j')P_{\ell+j}(x,y).
\]
Therefore
\[
|c_\theta(x,y)-c_{\theta'}(x,y)|
\le
C_{\mathrm{basis}}\|\theta-\theta'\|_1
\le
C_{\mathrm{basis}}\sqrt J\,\|\theta-\theta'\|_2
\le
L_{\mathrm{par}}\|\theta-\theta'\|_2.
\]
Taking the supremum over $(x,y)$ proves the first estimate. The second follows from the definition of
$\mathsf C_\theta$ by taking infima over $x\in B_x$.
\end{proof}

\begin{proposition}\label{prop:G-covering}
For every $\varepsilon>0$,
\[
\Cov\bigl(\varepsilon,\cG_Y,\|\cdot\|_{\infty,B_y}\bigr)
\le
\Cov\!\left(\frac{\varepsilon}{2L_{\mathrm{par}}},\Theta,\|\cdot\|_2\right)
\Cov\!\left(\frac{\varepsilon}{2},\cF_X,\|\cdot\|_{\infty,B_x}\right).
\]
\end{proposition}

\begin{proof}
Set
\[
\delta:=\frac{\varepsilon}{2L_{\mathrm{par}}},
\qquad
\eta:=\frac{\varepsilon}{2}.
\]
Choose a $\delta$-net $\mathcal A_{\Theta,\delta}\subset \Theta$ of minimal cardinality. For each $\theta\in \mathcal A_{\Theta,\delta}$,
Lemma~\ref{lem:Ttheta-contraction} shows that
\[
\Cov\bigl(\eta,\mathsf C_\theta(\cF_X),\|\cdot\|_{\infty,B_y}\bigr)
\le
\Cov\bigl(\eta,\cF_X,\|\cdot\|_{\infty,B_x}\bigr).
\]
Let $\mathcal H_{\theta,\eta}$ be an $\eta$-net for $\mathsf C_\theta(\cF_X)$ of this cardinality, and set
\[
\mathcal H_\eta:=\bigcup_{\theta\in \mathcal A_{\Theta,\delta}} \mathcal H_{\theta,\eta}.
\]
We claim that $\mathcal H_\eta$ is an $\varepsilon$-net of $\cG_Y$. Indeed, if $g=\mathsf C_\theta f\in \cG_Y$, choose
$\theta'\in \mathcal A_{\Theta,\delta}$ with $\|\theta-\theta'\|_2\le \delta$. By
Lemma~\ref{lem:parametric-lipschitz},
\[
\|\mathsf C_\theta f-\mathsf C_{\theta'}f\|_{\infty,B_y}
\le
L_{\mathrm{par}}\delta
=
\frac{\varepsilon}{2}.
\]
Then choose $h\in \mathcal H_{\theta',\eta}$ with
\[
\|\mathsf C_{\theta'}f-h\|_{\infty,B_y}\le \eta=\frac{\varepsilon}{2}.
\]
Hence
\[
\|g-h\|_{\infty,B_y}
\le
\|\mathsf C_\theta f-\mathsf C_{\theta'}f\|_{\infty,B_y}
+
\|\mathsf C_{\theta'}f-h\|_{\infty,B_y}
\le
\varepsilon.
\]
Therefore
\[
\Cov\bigl(\varepsilon,\cG_Y,\|\cdot\|_{\infty,B_y}\bigr)
\le
|\mathcal A_{\Theta,\delta}|\,
\Cov\bigl(\eta,\cF_X,\|\cdot\|_{\infty,B_x}\bigr)
\le
\Cov\bigl(\delta,\Theta,\|\cdot\|_2\bigr)
\Cov\bigl(\eta,\cF_X,\|\cdot\|_{\infty,B_x}\bigr).
\]
Using the choices of $\delta$ and $\eta$ gives the claim.
\end{proof}

\end{document}